\documentclass[a4paper]{amsart}

\usepackage{amssymb,amsthm,mathtools}
\usepackage[numbers]{natbib}
\usepackage[utf8]{inputenc}
\usepackage[T1]{fontenc}
\usepackage[shortlabels]{enumitem} % custom labels

\usepackage{xifthen}

\usepackage{commath} % for typesetting differentials

\usepackage{hyperref}
\usepackage{cleveref}

\usepackage{microtype}

\makeatletter
\AtBeginDocument{
  \hypersetup{
    pdftitle = {\@title},
  }
}
\makeatother

\newcommand{\R}{\mathbb{R}}
\newcommand{\N}{\mathbb{N}}

\newcommand{\F}{\mathcal{F}}
\newcommand{\G}{\mathcal{G}}
\newcommand{\B}{\mathcal{B}} % Borel \sigma-algebra
\newcommand{\T}[1][]{\ifthenelse{\isempty{#1}}{\mathbb{T}}{\mathbb{T}^{#1}}}

\renewcommand{\P}{\mathbb{P}}
\newcommand{\E}{\mathbb{E}}

\DeclareMathOperator{\Id}{Id}

\DeclareMathOperator{\vspan}{span}

\newcommand{\ds}{\dif s}   
\newcommand{\dt}{\dif t}   
\newcommand{\dW}{\dif W}    
\newcommand{\du} {\dif u}

\newcommand{\dx}{\dif x}
\DeclareMathOperator{\graph}{graph}

\newcommand{\divergence}{\nabla \!\! \cdot \!}
 \newcommand{\hess}[1][{}]{\operatorname{D}^{2}_{\!#1} \!}
\newcommand{\lapl}{\Delta}

\DeclareMathOperator{\supp}{supp}

\newcommand{\esssup}[1][]{
  \ifthenelse{\isempty{#1}}
  {}
  {{#1}\text{-}}
  \operatorname{esssup}
}

\newcommand{\sprod}[3][]{
  \ifthenelse{\isempty{#1}}
  {\left\langle #2, #3 \right\rangle}
  {\left\langle #2, #3 \right\rangle_{#1}}
}

\newcommand{\stochbasis}[1][{[0,T]}]{
  \ifthenelse{\equal{#1}{\infty}}
  {(\Omega, \F, (\F_t)_{t \in [0,\infty)}, \P)}
  {(\Omega, \F, (\F_t)_{t \in #1}, \P)}
}

\newcommand{\quadvar}[1]{\left\langle\!\left\langle #1 \right\rangle\!\right\rangle}

\DeclareMathOperator{\tr}{tr}

\newcommand{\QQ}[1][\nabla u]{\mathbf{Q}(#1)}
\newcommand{\QQu}[1][t]{\QQ[\nabla u(#1)]}
\newcommand{\vv}[1][\nabla u]{\mathbf{v}(#1)}
\newcommand{\vvu}[1][t]{\vv[\nabla u(#1)]}
\newcommand{\divv}[1][\nabla u]{\divergence \left(\vv[#1] \right)}
\newcommand{\divvu}[1][t]{\divv[\nabla u(#1)]}

\theoremstyle{plain}
\newtheorem{theorem}{Theorem}[section]
\newtheorem{corollary}[theorem]{Corollary}
\newtheorem{lemma}[theorem]{Lemma}
\newtheorem{proposition}[theorem]{Proposition}
\theoremstyle{definition}
\newtheorem{remark}[theorem]{Remark}

\newtheorem{assumptions}[theorem]{Assumptions}
\newtheorem{definition}[theorem]{Definition}

\allowdisplaybreaks

\title[Stochastic mean curvature flow of graphs]{Existence of martingale solutions and large-time behavior for a stochastic mean curvature flow of graphs}
\author{Nils Dabrock}
\address[N. Dabrock]{Fakultät für Mathematik, Technische Universität Dortmund, Vogelpothsweg 87, 44227 Dortmund, Germany}
\email{nils.dabrock@math.tu-dortmund.de}

\author{Martina Hofmanov\'a}
\address[M. Hofmanov\'a]{Faculty of Mathematics, Bielefeld University, Universitätsstraße 25, 33615 Bielefeld, Germany}
\email{hofmanova@math.uni-bielefeld.de}

\author{Matthias Röger}
\address[M. Röger]{Fakultät für Mathematik, Technische Universität Dortmund, Vogelpothsweg 87, 44227 Dortmund, Germany}
\email{matthias.roeger@math.tu-dortmund.de}

\date{\today}

\begin{document}

\begin{abstract}
  We are concerned with a stochastic mean curvature flow of graphs over a periodic domain of any space dimension. We  establish existence of martingale solutions which are strong in the PDE sense and study their large-time behavior. Our analysis is based on a viscous approximation and new global bounds, namely, an $L^{\infty}_{\omega,x,t}$ estimate for the gradient and an $L^{2}_{\omega,x,t}$ bound for the Hessian. The proof makes essential use of the delicate interplay between the deterministic mean curvature part and the stochastic perturbation, which permits to show that certain gradient-dependent energies are supermartingales. Our energy bounds in particular imply that solutions become asymptotically spatially homogeneous and approach a Brownian motion perturbed by a random constant.
  
\end{abstract}

%\subjclass[2010]{60H15, 60H30, 35Q30,
%76M35, 76N10}

\bibliographystyle{siam}

\maketitle

\setcounter{tocdepth}{1}
\tableofcontents

\section{Introduction}
The mean curvature flow (MCF) of hypersurfaces is one key example of a geometric evolution law and is of major importance both for applications and for the mathematical theory of surface evolution equations, see for example \cite{Zhu2002}, \cite{Ecker2004}, \cite{Mantegazza2011} or \cite{Bellettini2013} and the references therein.

Given a family $(\Gamma(t))_{t>0}$ of smooth $n$-dimensional hypersurfaces in $\R^{n+1}$ mean curvature motion is characterized by the evolution law
\begin{equation*}
    V(x,t) = H(x,t)\quad\text{ for }t>0,\,x\in \Gamma(t)
\end{equation*}
where $V$ describes the velocity in direction of a fixed smooth normal field $\nu$ and $H$ denotes the mean curvature with respect to the same normal field (in our notation $H$ is given by the sum of the principle curvatures).

The motion by mean curvature has attracted much attention. It is the simplest gradient flow dynamic of the surface area energy, that is a relevant energy in numerous applications. There are several analogies to the heat equation, as can be seen in the distance function formulation of MCF (see for example \cite{Bellettini2013}) or the approximation by mean curvature flow for nearly flat graphs. One of the consequences is that a comparison (or inclusion) principle holds and that convexity is conserved. On the other hand, MCF is a nonlinear evolution, governed by a degenerate quasilinear elliptic operator. This in particular leads to the possibility that singularities appear in finite time and that the topology changes. For example, balls shrink in finite time to points and for certain dumbbell type initial shapes a pinch-off of components happens.
Such challenges have been the origin and motivation for several important developments in geometric analysis, starting with the pioneering work of Brakke \cite{Brakke1978} on geometric measure theory approaches, level set methods as developed by Evans and Spruck \cite{Evans1991,Evans1992,Evans1992a,Evans1995} and Chen, Giga, Goto \cite{Chen1991}, De Giorgi's barrier method \cite{Bellettini1995,Bellettini1997} or time discrete approximations as introduced by Luckhaus and Sturzenhecker \cite{Luckhaus1995} and Almgren, Taylor and Wang \cite{Almgren1993}.

The formation of singularities on the other hand can be excluded in particular situations such as the evolution of entire graphs, where solutions exist globally in time \cite{Ecker1991} or for initial data given by compact, smooth and convex hypersurfaces \cite{Huisken1990}. In the latter case the surfaces become round and shrink to a point in finite time.

Several of the techniques developed for mean curvature flow have been successfully applied to deterministic perturbations of the flow \cite{Chen1991,Barles1993,Barles1998,Chambolle2008,Mugnai2011} that are present in a number of applications. A random forcing was included to mean curvature flow in \cite{Kawasaki1982} to account for thermal fluctuations.
In this paper we study a particular stochastic perturbation in the case of hypersurfaces given as graphs over the $n$-dimensional flat torus. To motivate the equation let us start from the general case of a random evolution $(\Gamma(t))_{t>0}$ of surfaces in $\R^{n+1}$ that are given by immersions $\phi_t:\Gamma\to\R^{n+1}$ of a fixed smooth manifold $\Gamma$. We then consider a real-valued Wiener process $W$ defined on some probability space $(\Omega,\mathcal{F},\P)$ and the stochastic differential equation
\begin{equation}
 \dif\phi_t(x)  =  \nu(x,t) \bigl(H(x,t) \dif t + \circ\dif W(t)\bigr),  \label{normalmodel}
\end{equation}
which is possibly the simplest stochastic perturbation, by a one dimensional white noise acting
uniformly in all points of the surface in normal direction.

If we further restrict ourselves to the case of graphs over the flat torus $\T^n$ (represented by the unit cube and periodic boundary conditions), that is,
\begin{equation*}
  \Gamma(\omega, t) = \graph u(\omega, \cdot, t) = \{ (x, u(\omega, x, t)) \in \R^{n+1} \mid x \in \T^n \}
\end{equation*}
for a (random) function $u: \Omega \times \T^n \times (0,\infty) \to \R$, we are lead to the following Stratonovich differential equation
\begin{align}
  \label{eq:smcf}
  \begin{split}
    \du &= \QQ \divv \dt + \QQ \circ \dW
  \end{split}
\end{align}
where $\QQ$ denotes the area element and $\vv$ the horizontal projection of the normal to the graph
\begin{align}
  \QQ[p] &\coloneqq \sqrt{1+|p|^2} \\
  \vv[p] &\coloneqq \frac{p}{\sqrt{1+|p|^2}}, ~p\in \R^n.
\end{align}

The choice of the Stratonovich instead of an It\^o differential in \eqref{eq:smcf} is necessary to keep the geometric character of the equation, see the discussion in \cite{Lions2002}.
Despite  its origin from a rather simple stochastic forcing, the evolution equation for the graphs presents severe difficulties. In particular, the presence of a multiplicative noise with nonlinear gradient dependence in combination with the degeneracy in the quasilinear elliptic term are challenges for a rigorous analysis and it is at first place not clear whether or not solutions stay graphs.

The deterministic mean curvature flow for graphs was considered in \cite{Ecker1989}, where an a priori gradient bound was proved and the long-time behavior was analyzed, see \cite{Huisken1989} for graphs over a given domain with vertical contact angle.
Lions and Souganidis presented a general well-posedness theory and introduced a notion of stochastic viscosity solutions \cite{Lions1998,Lions1998a,Lions2000a,Lions2000} for geometric equations of mean curvature flow type (and beyond), but %certain technical details of this approach are still being investigated \cite{MR1920103,MR2765508,MR3152786} and
no regularity properties other than continuity are obtained for the solutions. The evolution \eqref{normalmodel} for the case $n=1$ was investigated by Souganidis and Yip \cite{Souganidis2004} and Dirr, Luckhaus and Novaga \cite{Dirr2001}, where a stochastic selection principle was identified in situations where non-uniqueness appears for the deterministic flow. In \cite{Dirr2001} also an existence result was proved, but only for short time intervals determined by a random variable that is not necessarily bounded from below. Other (formal) approximations of stochastically perturbed mean curvature flow equations have been studied, such as a time discrete scheme in \cite{Yip1998} and stochastic Allen--Cahn equations in \cite{Funaki1995,Roeger2013,Weber2010,Bertini2017, Bertini2017a}.

The Stratonovich differential equation \eqref{eq:smcf} was already considered in the case $n=1$ in \cite{Es-Sarhir2012} and, mainly for $n=2$ in \cite{Hofmanova2017}, by von~Renesse and the second and third author. The present paper continues and extends these results in several respects. The most important new contribution is a uniform (i.e.~$L^\infty$ with respect to all the three variables $\omega, x, t$) gradient bound for $u$ and an $L^2$-bound for the Hessian in arbitrary dimensions. This is a major improvement compared to \cite{Hofmanova2017} where only $H^1$-estimates for $u$ and an $L^2$-estimate for the mean curvature were shown. Our gradient bound in particular shows that a solution stays a graph for all times. More precisely, Lipschitz continuity of the initial condition is preserved during the evolution.
As a consequence of our improved bounds we are able to prove the existence of martingale solutions that are strong in the PDE sense for any space dimension. In contrast, in \cite{Hofmanova2017} the existence result was restricted to two dimensions and the solutions were only weak in the PDE sense.

Our proof of the gradient bound uses a Bernstein type argument \cite[Section 14.1]{Gilbarg1977}  but in a context of energy methods, which seems to be new even for deterministic mean curvature flow equations. In the deterministic case this argument reduces to an argument which is similar to the way the gradient bounds in \cite{Ecker1989} are derived from Huisken's weighted monotonicity formula, but instead of the backward heat kernel a constant kernel is used.

Especially the $L^{\infty}$-gradient bound  and in particular its uniformity with respect to the randomness variable $\omega$ may appear somewhat surprising in the field of SPDEs. It is a consequence of the geometrical nature of the model and more precisely of the fact that the structure of the noise respects the  underlying deterministic evolution. This is reflected through our energy-type estimates: by exploring the precise structure of all the involved quantities we are able to group them in such a way that each term can be shown to be non-positive and additionally yields a control of second derivatives. The identification of the non-positive terms makes use of the interplay between the deterministic mean curvature part of the equation and the stochastic perturbation.

Moreover, we are also able to study the large-time behavior of solutions and prove that solutions become homogeneous in space and asymptotically only deviate from a constant value by a Wiener process. This result improves the results of \cite{Es-Sarhir2012} by obtaining a stronger convergence and extending it to arbitrary dimensions.

In contrast to \cite{Hofmanova2017} we will use the abstract theory of variational SPDEs \cite{Prevot2007} to handle equation \eqref{eq:smcf}. Although \eqref{eq:smcf} itself only has a variational structure for $n = 1$, which was exploited in \cite{Es-Sarhir2012}, the gradient of a solution will indeed solve a variational SPDE for arbitrary dimensions. Since \eqref{eq:smcf} lacks coercivity we will approximate it for $\varepsilon > 0$ by
\begin{align}
  \label{eq:smcf-viscous}
  \begin{split}
    \du &= \varepsilon \lapl u + \QQ \divv \dt + \QQ \circ \dW,
  \end{split}
\end{align}
which is coercive in an appropriate sense. We will call \eqref{eq:smcf-viscous} the viscous equation.

Since the viscous equation is not covered by the classical theory for variational SPDEs \cite{Prevot2007, Gawarecki2011}, we  include an It\^o formula and an abstract existence result for a large class of equations in \Cref{sec:variational-spdes}. These results, which hold independent interest themselves, are generalizations of the results from the pioneering works of Pardoux \cite{Pardoux1975} and Viot \cite{Viot1976}.

For a precise formulation of our main results and an overview over the main techniques of the proofs see \Cref{sec:results} below.

We note that under our assumptions on the initial condition the stochastic viscosity theory \`a la Lions, Souganidis  \cite{Lions1998,Lions1998a,Lions2000a,Lions2000} yields the existence of a unique viscosity solution. Proving the coincidence of our solution with the viscosity solution seems a major challenge and out of reach at the moment. Comparing the two notions, our solutions have better regularity properties implying not only space-time H\"older continuity but in addition $L^2$-regularity of second order derivatives in space. In particular the mean curvature operator is well-defined in a pointwise a.e. sense. Furthermore, we are able to characterize the large-time behavior. On the other hand, proving uniqueness for our solutions (which is necessary and most likely also a major tool to obtain the  equivalence of the concepts) remains open. For our solutions an energy based approach to uniqueness seems most appropriate but to require even higher regularity of solutions and a control of the evolution of quantities like the normal vectors or the surface area measure.

This paper is organized as follows: After explaining the notation in \Cref{sec:notation} we present our results in \Cref{sec:results}. Existence of solutions of the viscous equation will be established in \Cref{sec:existence-viscous}. In \Cref{sec:a-priori-estimates} we prove similarly to Huisken's monotonicity formula that certain energies are non-increasing uniformly in $\varepsilon$. We apply this to deduce uniform $H^2$ and uniform $L^\infty$ gradient bounds for solutions of the viscous equation. In \Cref{sec:vanishing-viscosity} we  prove that solutions of \eqref{eq:smcf-viscous} converge to a solution of \eqref{eq:smcf}, which in particular proves that there exists a solution. The large-time behavior of a solution is analyzed in \Cref{sec:large-time-behavior}.

We present the theory of variational SPDEs in spaces with compact embedding in \Cref{sec:variational-spdes}.

%\rmk{streamlined presentation, longer appendix with all the technical results. therefore it is possible to read the arguments in the main part without technicalities.}

\section{Notation}\label{sec:notation}
In this section we introduce the basic notation used throughout the paper.

\subsection*{Hilbert-Schmidt operators}
  Let $U, H$ be two separable Hilbert spaces and $(g_k)_{k}$ an orthonormal basis of $U$. With $L_2(U;H)$ we will denote the space of all Hilbert-Schmidt operators $T : U \to H$ with the norm $\| T \|_{L_2(U;H)}^2 \coloneqq \sum_{k} \| T g_k \|_{H}^2$, which is independent of the choice of the orthonormal basis.

  With $L_1(H)$ we will denote the space of all nuclear operators $T : H \to H$ with the norm
  \begin{align*}
    \| T \|_{L_1(H)} \coloneqq \inf \left\{ \sum_{k=1}^\infty \| a_k \|_H \| \varphi_k \|_H \mid (a_k)_k \subset H,\, (\varphi_k)_k \subset H', \, T = \sum_{k=1}^\infty a_k \varphi_k \right\}.
  \end{align*}
  It is well known that $(L_1(H))^\ast = L(H)$ and that the weak-$\ast$ topology on $L(H)$ coincides on norm bounded subsets with the weak operator topology on $L(H)$, which is the weakest topology such that for all $x,y \in H$ the map $L(H) \to \R, T \mapsto \sprod[H]{Tx}{y}$ is continuous.

  Furthermore, for a Banach space $E$ we will use the notation $(E,w)$ resp. $(E', w^\ast)$ to denote the space $E$ with the weak topology resp. the dual space $E'$ with the weak-$\ast$ topology.

\subsection*{Stochastic processes}
  For an interval $I = [0,T]$ with $T > 0$ or $I = [0,\infty)$, a stochastic basis $(\Omega, \F, (\F_t)_{t \in I}, \P)$ consists of a probability space $(\Omega; \F, \P)$ together with a filtration $(\F_t)_{t \in I}$. According to \cite{DaPrato2014} the filtration $(\F_t)_t$ will be called a normal filtration, if
  \begin{itemize}
  \item $A \in \F_0$ for all $A \in \F$ with $\P(A) = 0$ and
  \item for all $t \in I$ with $t < \sup I$ we have that
    \begin{align*}
      \F_t = \bigcap_{s > t} \F_s.
    \end{align*}
  \end{itemize}
  A Wiener process $W = (W_t)_{t \in I}$ with respect to the probability space $(\Omega, \F, \P)$ is called a $(\F_t)_{t \in I}$-Wiener process if
  \begin{itemize}
  \item $W_t$ is $\F_t$ measurable for all $t \in I$ and
  \item $W_t - W_s$ is independent of $\F_s$ for all $s, t \in I$ with $s < t$.
  \end{itemize}
  For such an $(\F_t)_t$-Wiener process $W$ on a separable Hilbert space $U$ with covariance operator $Q \in L(U)$, that we always assume to be positive definite, one can define the space $U_0 \coloneqq Q^{\frac{1}{2}}(U)$ with the induced scalar product $\sprod[U_0]{x}{y} \coloneqq \sprod[U]{Q^{-\frac{1}{2}}x}{Q^{-\frac{1}{2}}y}$. If $H$ is another separable Hilbert space and $\Phi$ is a predictable $L_2^0 \coloneqq L_2(U_0;H)$-valued process with
  \begin{align*}
    \P\left( \int_0^T \| \Phi(t) \|_{L_2^0}^2 \dt < \infty \right) = 1,
  \end{align*}
  then the stochastic It\^o integral
  \begin{align*}
    \int_0^t \Phi(s) \dW_s, ~t \in [0,T]
  \end{align*}
  is a well-defined local martingale with values in $H$.
  
\subsection*{Stratonovich integral}
  \label{rem:stratonovich}
  In the situation above, it is sometimes more natural to consider the stochastic Stratonovich integral
  \begin{align*}
    \int_0^t \Phi(s) \circ \dW_s, ~t \in [0,T],
  \end{align*}
  which, however,  might not be well-defined.

  If at least formally one has the evolution law
  \begin{align*}
    \dif \Phi = \mu \dt + \sigma \circ \dW,
  \end{align*}
  with an $L_2^0$-valued process $\mu$ and an $L_2(U_0; L_2^0) = L_2(U_0 \times U_0 ; H)$-valued process $\sigma$,
  then formally one has
  \begin{align}
    \label{eq:stratonovich}
    \int_0^t \Phi(s) \circ \dW_s = \int_0^t \Phi(s) \dW_s + \frac{1}{2} \int_0^t \left[ \sum_{k} \left(\sigma(s)Q^{\frac{1}{2}}g_k\right)Q^{\frac{1}{2}}g_k \right] \ds
  \end{align}
  for all $t \in [0,T]$, with $(g_k)_k$ an orthonormal basis of $U$. The value on the right hand side does not depend on the choice of $(g_k)_k$.

  Whenever the right hand side of \eqref{eq:stratonovich} is well-defined, we can think of it as the definition for the Stratonovich integral on the left hand side of \eqref{eq:stratonovich}.

  We will call
  \begin{align*}
    \frac{1}{2}  \sum_{k} \left(\sigma(t) Q^{\frac{1}{2}}g_k\right)Q^{\frac{1}{2}}g_k
  \end{align*}
  the It\^o-Stratonovich correction term.

\subsection*{Periodic Sobolev spaces}
  For $k \ge 0$, $p \in [1,\infty]$ we will denote with $W^{k,p}(\T^n)$ the space of periodic Sobolev functions on the flat torus $\T^n$, which can be identified with the completion of the space of $[0,1]^n$ periodic $C^\infty(\R^n)$ functions with respect to the $\| \cdot \|_{W^{k,p}([0,1]^n)}$ norm.

\subsection*{Matrix scalar product}
  For matrices $A,B,C,D \in \R^{n\times n}$ we will write
  \begin{align*}
    A:B \coloneqq \sum_{i,j=1}^n A_{ij} B_{ij}.
  \end{align*}
  We will use the convention that
  \begin{align*}
    AB : CD \coloneqq (AB) : (CD) = \sum_{i,j,k,l=1}^n A_{ij}B_{jk}C_{il}D_{lk}.
  \end{align*}

\section{Results}\label{sec:results}
In this section we will state the main results of this paper. The proofs are given in the subsequent sections.
We will first formulate our solution concept. We are concerned with solutions that are strong in the PDE sense, that is, an integral form of \eqref{eq:smcf} is satisfied pointwise. In addition, they may be either strong or weak in the probabilistic sense, depending on whether the underlying probabilistic elements are given in advance or not.

\begin{definition}\ 
  \label{defn:solution}
  \begin{enumerate}[(i)]
  \item \label{defn:solution_i} Let $I = [0,\infty)$ or $I = [0,T]$ with $T > 0$, $\stochbasis[I]$ be a stochastic basis with a normal filtration together with a real-valued $(\F_t)$-Wiener process $W$ and $u_0 \in L^2(\Omega; H^1(\T^n))$ be $\F_0$-measurable.  A predictable $H^2(\T^n)$-valued process $u$ with $u \in L^2(\Omega; L^2(0,t;H^2(\T^n)))$ for all $t \in I$ is a strong solution of \eqref{eq:smcf} with initial data $u_0$, if
    \begin{align*}
      u(t) - u_0 &= \int_0^t \QQ[\nabla u(s)] \divvu[s] \ds \\
                 &\phantom{{}={}} + \int_0^t \QQ[\nabla u(s)] \circ \dW_s \quad \P\text{-a.s.} \text{ in } L^2(\T^n) ~\forall t \in I.
    \end{align*}
  \item Let $\Lambda$ be a Borel probability measure on $H^1(\T^n)$ with bounded second moments $\int_{H^1(\T^n)} \| z \|_{H^1(\T^n)}^2 \dif \Lambda(z) < \infty$.  A martingale solution of \eqref{eq:smcf} with initial data $\Lambda$ is given by $\stochbasis[I]$ together with $W$, $u_0$ and $u$ such that \ref{defn:solution_i} is satisfied and $\P \circ u_0^{-1} = \Lambda$.
  \end{enumerate}
  In the same way we can define strong solutions and martingale solutions for \eqref{eq:smcf-viscous}.
\end{definition}

In the following we will often just write that $u$ is a strong solution instead of specifying that $u$ is a strong solution for a time interval $I$ with respect to a stochastic basis with a normal filtration and a real-valued Wiener process. If not otherwise specified the stochastic basis will be denoted by $\stochbasis[I]$ and the Wiener process by $W$.

\begin{remark}
  Note that formally for a solution $u$ of \eqref{eq:smcf} one can use the chain rule, which holds true for the Stratonovich integral, to deduce that
  \begin{align*}
    \dif \left( \QQ \right) = \vv \cdot \nabla \left( \QQ \divv \right) \dt + \vv \cdot \nabla \left( \QQ \right) \circ \dW.
  \end{align*}

  Hence, according to \Cref{rem:stratonovich} the It\^o-Stratonovich correction for the integral in \Cref{defn:solution} is given by
  \begin{align*}
    \frac{1}{2} \vv \cdot \nabla \left( \QQ \right) = \frac{1}{2} \vv \cdot \hess u \vv
  \end{align*}
  and the Stratonovich integral in \Cref{defn:solution} has to be understood in the sense that
  \begin{align*}
    \int_0^t \QQu[s] \circ \dW_s &\coloneqq \int_0^t \QQu[s] \dW_s \\
                                 &\quad+ \frac{1}{2} \int_0^t \vvu[s] \cdot \hess u(s) \vvu[s] \ds,
  \end{align*}
  such that the equation in \Cref{defn:solution} becomes
  \begin{align*}
    u(t) - u_0 &= \int_0^t \left[\QQu[s] \divvu[s] + \frac{1}{2} \vvu[s] \cdot \hess u(s) \vvu[s] \right] \ds \\
               &\phantom{{}={} } + \int_0^t \QQu[s] \dW_s \\
               &= \int_0^t \left[\lapl u(s) - \frac{1}{2} \vvu[s] \cdot \hess u(s) \vvu[s] \right] \ds\\
               &\quad + \int_0^t \QQu[s] \dW_s.
  \end{align*}
\end{remark}

\begin{remark}
  Note that for a strong  solution  the stochastic basis and the Wiener process are prescribed, whereas for a martingale solution, i.e. probabilistically weak solution, the stochastic basis and the Wiener process are part of the solution. Once this stochastic basis and the corresponding Wiener process are found, the martingale solution is a strong solution with respect to this particular choice of stochastic basis and Wiener process.
\end{remark}

\begin{remark}
  From \Cref{cor:variational-sobolev-ito} we infer that a strong solution of \eqref{eq:smcf} or \eqref{eq:smcf-viscous} has a modification with continuous paths in $H^1(\T^n)$ and $u \in L^2(\Omega; C([0,t]; H^1(\T^n)))$ for all $t \in I$. 
  Furthermore, under suitable assumptions on the initial data we deduce that $u \in C([0,t]; C(\T^n))$ $\P$-a.s. for all $t \in I$, see \Cref{rem:regularity-of-solution} below.
\end{remark}

We are now ready to state the main result of the present paper.
\begin{theorem}[Existence of martingale solutions]
  \label{thm:main}
  Let $\Lambda$ be a Borel probability measure on $H^1(\T^n)$ with bounded second moments and additionally
  \begin{align*}
    \supp \Lambda \subset \{ z \in H^1(\T^n) \mid \| \nabla z \|_{L^\infty(\T^n)} \le L \}
  \end{align*}
  for some constant $L > 0$.

  Then for $I = [0,\infty)$ there is a martingale solution of \eqref{eq:smcf} with initial data $\Lambda$. For all such solutions it holds that $\hess u \in L^2(\Omega; L^2(0,\infty;L^2(\T^n)))$ and
  \begin{align*}
    \| \nabla u \|_{L^\infty(0,\infty; L^\infty(\T^n))} \le L \quad \text{a.s.}
  \end{align*}
\end{theorem}

Our next main result shows that solutions become spatially constant for $t\to\infty$.
\begin{theorem}[Large-time behavior]
  \label{thm:large-time}
  Let the assumptions from \Cref{thm:main} hold and $u$ be a martingale solution of \eqref{eq:smcf} for $I = [0,\infty)$.

  Then there is a real-valued random variable $\alpha$ such that
  \begin{align*}
    \E \sup_{t \ge T} \left\| u(t) - W(t) - \alpha \right\|_{H^1(\T^n)} \to 0 \quad \text{for } T \to \infty.
  \end{align*}
\end{theorem}

\begin{remark}
  We will deduce existence of solutions $(u^\varepsilon)_{\varepsilon > 0}$ of the viscous equation \eqref{eq:smcf-viscous} using the abstract theory of variational SPDEs presented in \Cref{sec:variational-spdes}. The fact that \eqref{eq:smcf-viscous} can be treated as a coercive equation already yields estimates for the Dirichlet energy of solutions.

  In \Cref{sec:a-priori-estimates} we will extend these arguments to prove more general a priori estimates for solutions which are uniform in $\varepsilon > 0$. For this we will make use of a generalization of the classical It\^o formula to prove that certain gradient-dependent energies are non-increasing for solutions in a stochastic sense, i.e. they are supermartingales. In the deterministic case one can use Huisken's monotonicity formula to get similar results. With the stochastic perturbation, Huisken's monotonicity formula does not hold because the time-derivative of these energies contains additional It\^o-Stratonovich correction terms that are difficult to control. For our gradient-dependent energies we use integration by parts to prove that these correction terms together with terms stemming from the deterministic motion have a good sign. We will apply this result to deduce estimates for the Dirichlet energy in \Cref{prop:dirichlet} and a maximum principle for the gradient in \Cref{prop:maximum-principle}.

  With our uniform Lipschitz bounds at hand and \Cref{prop:dirichlet} we deduce that \eqref{eq:smcf} is coercive and this yields $H^2$ bounds for $(u^\varepsilon)$. Furthermore we derive tightness of their probability laws in appropriate spaces and with the Jakubowski-Skorokhod representation we can deduce that the approximate solutions converge in a weak sense. We then identify the limit in \Cref{sec:vanishing-viscosity}.

  The a priori estimates derived for the solution are also one key to analyze the large-time behavior of solutions. 
\end{remark}

\section{Existence of viscous approximation}
\label{sec:existence-viscous}
We will use the theory presented in \Cref{sec:variational-spdes} to prove existence for a viscous approximation \eqref{eq:smcf-viscous} of the stochastic mean curvature flow. The key observation is that the variational framework shall be applied to the equation for $\nabla u$, see \eqref{eq:smcf-viscous-gradient} below, rather than directly to \eqref{eq:smcf-viscous}. This is further made possible by the structure of \eqref{eq:smcf-viscous} and in particular by the fact that only the gradient of the solution appears on the right hand side of \eqref{eq:smcf-viscous}.

\begin{theorem}
  \label{thm:existence_viscous}
  Let $\varepsilon > 0$, $q > 2$ and $\Lambda$ be a Borel probability measure on $H^1(\T^n)$ with 
  \begin{align*}
    \int_{H^1(\T^n)} \| z \|_{H^1(\T^n)}^2 \dif \Lambda(z) < \infty \text{ and } \\
    \int_{H^1(\T^n)} \| \nabla z \|_{L^2(\T^n)}^q \dif \Lambda(z) < \infty.
  \end{align*}
  Then there is a martingale solution $u$ of \eqref{eq:smcf-viscous} for $I = [0,\infty)$ with initial data $\Lambda$.

  \begin{proof}[Proof of \Cref{thm:existence_viscous}]
    We intend to apply \Cref{thm:variational-existence} in order to obtain a martingale solution to the equation the gradient $\nabla u$ fulfills for $u$ satisfying \eqref{eq:smcf-viscous}, which in turn yields a martingale solution to \eqref{eq:smcf-viscous} itself.
    To this end, we will work with the spaces
    \begin{align*}
      V &\coloneqq \{ \nabla u  \mid u \in H^2(\T^n) \} \text{ with } \| \nabla u \|_V \coloneqq \| \nabla u \|_{H^1(\T^n; \R^n)}, \\
      H &\coloneqq \{ \nabla u \mid u \in H^1(\T^n) \} \text{ with } \| \nabla u \|_H \coloneqq \| \nabla u \|_{L^2(\T^n; \R^n)} \text{ and } \\
      U &\coloneqq \R.
    \end{align*}
    We have that $V \subset H$ densely and compactly. Furthermore we can identify $L_2(U;H) = H$.

    We define the operators
    \begin{gather*}
      A_\varepsilon : V \to V' \\
      \begin{split}
        &\sprod[V',V]{A_\varepsilon(\nabla u)}{\nabla w} \\
        &\quad\coloneqq - \int_{\T^n} \left( \varepsilon \Delta u + \QQ\divv + \frac{1}{2} \vv \cdot \hess u \vv \right) \Delta w \\
        &\quad= - \int_{\T^n} \left( (1+\varepsilon) \Delta u - \frac{1}{2} \vv \cdot \hess u \vv \right) \Delta w
      \end{split}
    \end{gather*}
    and
    \begin{gather*}
      B : V \to H \\
      B(\nabla u) \coloneqq \nabla \left( \QQ \right) = \hess u \vv.
    \end{gather*}

    We verify that the \Cref{assumptions:variational-existence} are fulfilled:
    \begin{itemize}
    \item \textbf{Coercivity:} Using integration by parts and the fact that the boundary terms vanish because of the periodic domain we obtain
      \begin{align*}
        2&\sprod[V',V]{A_\varepsilon(\nabla u)}{\nabla u} + \| B(\nabla u) \|_H^2 \\
         &= \int_{\T^n} -\left( 2\varepsilon \Delta u + 2\QQ \divv +  \vv \cdot \hess u \vv \right) \Delta u \\
        &\phantom{{}={}\int} + |\hess u \vv|^2 \\
                                                         &= \int_{\T^n} - 2\varepsilon (\Delta u)^2 - \frac{1}{2} |\QQ \divv|^2- (\Delta u)^2 + |\hess u \vv|^2 \\
                                                         &\phantom{{}={}\int} + \frac{1}{2} \divv \left(\QQ^2 \divv -  2\QQ \Delta u\right)  \\
                                                         &= \int_{\T^n} -2\varepsilon (\Delta u)^2 - \frac{1}{2} |\QQ \divv |^2 \\
                                                         &\phantom{{}={}\int}- \frac{1}{2} \divv \divergence \left( Q\nabla u \right) - |\hess u|^2 + |\hess u \vv|^2 \\
                                                         &= \int_{\T^n} -2\varepsilon (\Delta u)^2 - \frac{1}{2} |\QQ \divv |^2 \\
                                                         &\phantom{{}={}\int}- \frac{3}{2} |\hess u|^2 + |\hess u \vv|^2 + \frac{1}{2} |\vv \cdot \hess u \vv|^2 \\
                                                         &\le -2 \varepsilon \| \Delta u \|_{L^2(\T^n)}^2 \\
                                                         &\le -C \varepsilon \| \nabla u \|_{H^1(\T^n; \R^n)}^2.
      \end{align*}
      Note that we have used the non-negativity of
      \begin{align*}
        \frac{3}{2}|\hess u|^2 - |\hess u \vv|^2 - \frac{1}{2} |\vv \cdot \hess u \vv |^2
      \end{align*}
      in the second to last inequality and the periodic boundary conditions as well as a Poincar\'e inequality for mean-free vector fields in the last inequality.
    \item \textbf{Growth bounds:} We have
      \begin{align*}
        \| A_\varepsilon(\nabla u) \|_{V'}^2 &\le \int_{\T^n} \left| \varepsilon \Delta u +\QQ \divv + \frac{1}{2} \vv \cdot \hess u \vv \right|^2 \\
                                             &= \int_{\T^n} \left| (1 + \varepsilon)\Delta u - \frac{1}{2}\vv \cdot \hess u \vv \right|^2 \\
                                             &\le C \| \nabla u \|_{H^1(\T^n;\R^n)}^2, \\
        \| B(\nabla u) \|_H^2 &= \| \nabla (\QQ) \|_{L^2(\T^n;\R^n)}^2 \le C \| \nabla u \|_{H^1(\T^n;\R^n)}^2, \\
        \| B(\nabla u) \|_{V'}^2 &\le C \left(1 + \| \nabla u \|_{L^2(\T^n;\R^n)}^2\right).
      \end{align*}
    \item \textbf{Continuity:} When $\nabla u_k \rightharpoonup \nabla u$ in $V$, then $\nabla u_k \to \nabla u$ in $H$ and therefore
      \begin{align*}
        \vv[\nabla u_k] \cdot \hess u_k \vv[\nabla u_k] &= \vv[\nabla u_k] \otimes \vv[\nabla u_k] : \hess u_k \\
                                                        &\rightharpoonup \vv \otimes \vv : \hess u \\
                                                        &= \vv \cdot \hess u \vv \text{ in } L^1(\T^n)
      \end{align*}
      and since $|\vv[\nabla u_k]| \le 1$ also
      \begin{align*}
        \vv[\nabla u_k] \cdot \hess u_k \vv[\nabla u_k]\rightharpoonup \vv \cdot \hess u \vv \text{ in } L^2(\T^n).
      \end{align*}
      The other terms in the definition of $A(u_k)$ are linear in $u_k$, hence
      \begin{align*}
        A(u_k) \overset{\ast}{\rightharpoonup} A(u) \text{ in } V'.
      \end{align*}
      Similarly
      \begin{align*}
        B(\nabla u_k) = \hess u_k \vv[\nabla u_k] \rightharpoonup \hess u \vv \text{ in } L^2(\T^n;\R^n).
      \end{align*}
    \end{itemize}

      Now, from \Cref{thm:variational-existence} we can conclude that there is a martingale solution $\nabla u$ of
      \begin{align}
        \label{eq:smcf-viscous-gradient}
        \begin{split}
          \dif \nabla u &= \nabla \left(\varepsilon \lapl u + \QQ \divv + \frac{1}{2} \vv \cdot \hess u \vv \right) \dt \\
          &\quad + \nabla \left(\QQ\right) \dW \\
          &= \nabla \left(\varepsilon \lapl u + \QQ \divv \right) \dt + \nabla \left(\QQ\right) \circ \dW \text{ in } V'
        \end{split}
      \end{align}
      with a real-valued Brownian motion $W$.

      Next we will show that \eqref{eq:smcf-viscous-gradient} is also fulfilled in $H^{-1}(\T^n;\R^n)$, hence weak in the PDE sense. 
      For an arbitrary $\psi \in H^1(\T^n;\R^n)$ we take the Helmholtz decomposition $\psi = \nabla w + \phi$ with $w \in H^2(\T^n)$ and $\phi \in H^1(\T^n;\R^n)$ with $\divergence \phi = 0$ and since both sides of the equation for $\nabla u$ are orthogonal to divergence-free vector fields, we have for all $t \in [0, \infty)$
      \begin{align*}
        \int_{\T^n} \left(\nabla u(t) - \nabla u_0\right) \cdot \psi
        &=\int_0^t \sprod[H^{-1};H^1]{\nabla \left(\QQu[s] \divvu[s] \right)}{\psi} \ds \\
        &\phantom{{}={}} + \int_0^t \int_{\T^n} \nabla \left(\QQu[s]\right) \cdot \psi \circ \dW(s)
      \end{align*}
      and therefore the equation for $\nabla u$ is also fulfilled in $H^{-1}(\T^n;\R^n)$.
 
      Now, define for $t \in [0,\infty)$ 
      \begin{align}
        \label{eq:existence-viscous-tilde-u}
        \tilde{u}(t) \coloneqq u_0 + \int_0^t  \QQu[s] \divvu[s] \ds + \int_0^t \QQu[s] \circ \dW(s).
      \end{align}
      Note that by assumption $u_0 \in L^2(\Omega; L^2(\T^n))$ and also for $T \in [0,\infty)$
      \begin{align*}
        t &\mapsto \int_0^t \QQu[s] \divvu[s] \ds \in L^2(\Omega; L^2(0,T; L^2(\T^n))) \text{ and } \\
        t &\mapsto  \int_0^t \QQu[s] \circ \dW(s)\in L^2(\Omega; L^2(0,T; L^2(\T^n))) .
      \end{align*}
      Hence, $\tilde{u} \in L^2(\Omega; L^2(0,T; L^2(\T^n)))$. Furthermore
      \begin{align*}
        \nabla \tilde{u}(t) &= \nabla u_0 + \int_0^t \nabla \left(\QQu[s] \divvu[s] \right)\ds\\
                            &\quad + \int_0^t \nabla \left(\QQu[s] \right) \circ \dW(s) = \nabla u(t) ~\forall t \in [0,\infty) ~\P\text{-a.s.}
      \end{align*}
      and by \eqref{eq:existence-viscous-tilde-u} $\tilde{u}$ is a martingale solution of \eqref{eq:smcf-viscous}.
  \end{proof}
\end{theorem}

% \section{Proof of \texorpdfstring{\Cref{prop:dirichlet}}{Proposition \ref{prop:dirichlet}}}
%\label{sec:proof-dirichlet}
\section{A priori estimates}
\label{sec:a-priori-estimates}

In this section we will prove a priori energy estimates for solutions of the viscous equation \eqref{eq:smcf-viscous} which are uniformly in $\varepsilon > 0$ and also hold true for solutions of the SMCF equation \eqref{eq:smcf}.
The first proposition basically says that the Dirichlet energy of solutions is decreasing and extends the coercivity proven in \Cref{sec:existence-viscous}.

\begin{proposition}[Weak coercivity]
  \label{prop:dirichlet}
  Let $\varepsilon \ge 0$ and $u$ be a strong solution of \eqref{eq:smcf-viscous}.  Then the energy $\int_{\T^n} |\nabla u|^2$ is a supermartingale.

  Furthermore, we can quantify the decay by
  \begin{align*}
    \E \| &\nabla u(t) \|_{L^2(\T^n)}^{2} + 2 \varepsilon \E\int_0^t \int_{\T^n} |\hess u(s)|^2 \ds \\
                                         &+ \frac{1}{2} \E \int_0^t \int_{\T^n} \QQ[\nabla u(s)]^2 |\divvu[s]|^2 \ds \\
                                         &+ \E \int_0^t \int_{\T^n} \left(\frac{3}{2}|\hess u|^2 - |\hess u \vv|^2 - \frac{1}{2} |\vv \cdot \hess u \vv |^2 \right)(s) \ds \\
                                         &\le \E \| \nabla u_0 \|_{L^2(\T^n)}^2 ~\forall t \in I.
  \end{align*}
  Note that $\frac{3}{2}|\hess u|^2 - |\hess u \vv|^2 - \frac{1}{2} |\vv \cdot \hess u \vv |^2 \ge 0$.

  We also have for $q \in [1,2)$ with a universal constant $C$, that
  \begin{align*}
    \E \sup_{t \in I} \| \nabla u(t) \|_{L^2(\T^n)}^{2q} \le \left(2 +  \frac{2C^2}{2q - q^2}\right) \E  \| \nabla u_0 \|_{L^2(\T^n)}^{2q}.
  \end{align*}

  If in addition $\esssup[\P] \| \nabla u \|_{L^\infty(I;L^\infty(\T^n))} = L < \infty$, then we have
  \begin{align*}
    \E \| \nabla u(t) \|_{L^2(\T^n)}^{2} &+ \frac{3 + 4L^2}{2(1+L^2)^2} \E \int_0^t \int_{\T^n} |\hess u(s)|^2 \ds \le \E \| \nabla u_0 \|_{L^2(\T^n)}^2 ~\forall t \in I.
  \end{align*}
\end{proposition}

In the next proposition we prove that the additional assumptions from \Cref{prop:dirichlet} can be verified if the Lipschitz constant of the initial condition is uniformly bounded.

\begin{proposition}[Maximum principle for the gradient of solutions]
  \label{prop:maximum-principle}
  Let $\varepsilon \ge 0$ and $u$ be a strong solution of \eqref{eq:smcf-viscous}.  If $\esssup[\P] \| \nabla u_0 \|_{L^\infty(\T^n)}< \infty$ then $\nabla u \in L^\infty(I; L^\infty(\T^n))$ a.s. with
  \begin{equation*}
    \| \nabla u \|_{L^\infty(I; L^\infty(\T^n))} \le \esssup[\P] \| \nabla u_0 \|_{L^\infty(\T^n)} ~\text{a.s.}
  \end{equation*}
\end{proposition}

\Cref{prop:dirichlet} and \Cref{prop:maximum-principle}  are proved at the end of this section. Both are based on an It\^o formula for integrals of the gradient of solutions. We summarize this calculation in the next lemma.
\begin{lemma}
  \label{lem:energy_ito_formula}
  Let $\varepsilon \ge 0$ and $u$ be a strong solution of \eqref{eq:smcf-viscous}.  For a function $f \in C^2(\R^n)$ with bounded second order derivatives and
  \begin{equation*}
    \mathcal{I}(t) \coloneqq \int_{\T^n} f(\nabla u(t)), ~t\in I
  \end{equation*}
  we obtain
  \begin{align*}
    \begin{split}
      \dif \mathcal{I} &= \int_{\T^n} -\varepsilon \hess f(\nabla u) \hess u : \hess u
      + \int_{\T^n} - \frac{1}{2} f(\nabla u) |\divv|^2 \\
      & + \int_{\T^n}  \hess u \left(\Id - \vv \otimes \vv\right) \\&\quad\quad\quad: \left(\frac{f(\nabla u)}{2 \QQ^2} \left(\Id - \vv\otimes \vv\right) -\hess f(\nabla u)  \right) \hess u \dt \\
      & - \int_{\T^n} f(\nabla u)\divv \dW.
    \end{split}
  \end{align*}

  \begin{proof}
    To abbreviate the calculations we will write $Q \coloneqq \QQ$ and $v \coloneqq \vv$. With this notation we have $\nabla Q = \nabla \left(\QQ \right) = \hess u \vv = \hess u v$.  We can apply \Cref{cor:variational-sobolev-ito} to infer
    \begin{align}
      \label{eq:lemma_dI}
      \begin{split}
        \dif \mathcal{I} &= \int_{\T^n} - \hess f(\nabla u) : \hess u \left( \varepsilon \lapl u + Q \divergence v + \frac{1}{2} v \cdot
          \nabla Q \right) \dt \\
        &\phantom{{}={}} + \int_{\T^n} \frac{1}{2} \nabla Q \cdot \hess f(\nabla u) \nabla Q \dt \\
        &\phantom{{}={}} + \int_{\T^n} \nabla f(\nabla u) \cdot \nabla Q \dW \eqqcolon \varepsilon \mu_{\text{viscous}} + \frac{1}{2} \mu_{\text{mcf}} + \frac{1}{2} \mu_{\text{pert}} \dt + \sigma \dW
      \end{split}
    \end{align}
    with
    \begin{align*}
      \mu_{\text{viscous}} &= \int_{\T^n} -\hess f(\nabla u) : \hess u \lapl u, \\
      \mu_{\text{mcf}} &= \int_{\T^n} -\hess f(\nabla u) : \hess u Q \divergence v, \\
      \mu_{\text{pert}} &= \int_{\T^n} -\hess f(\nabla u) : \hess u \lapl u + \nabla Q \cdot \hess f(\nabla u) \nabla Q, \\
      \sigma &= \int_{\T^n} \nabla f(\nabla u) \cdot \nabla Q.
    \end{align*}
    The term $\mu_{\text{viscous}}$ corresponds to the time derivative of $\mathcal{I}$ along solutions of the heat equation. It is weighted with $\varepsilon$ because it appears due to the additional viscosity added to the equation. The term $\mu_{\text{mcf}}$ corresponds to the time derivative of $\mathcal{I}$ along solutions of the unperturbed mean curvature flow of graphs. It is weighted with the factor $\frac{1}{2}$ because the other part has to be used in $\mu_{\text{pert}}$ to handle the additional terms coming from the perturbation. We handle $\mu_{\text{viscous}}$, $\mu_{\text{mcf}}$ and $\mu_{\text{pert}}$ separately using partial integrations and the periodicity of the functions. For $\mu_{\text{viscous}}$ we calculate
    \begin{align*}
      \mu_{\text{viscous}} &= \int_{\T^n} -\divergence \left( \nabla f(\nabla u) \right)  \lapl u \\
                           &= \int_{\T^n} - \hess f(\nabla u) \hess u : \hess u.
    \end{align*}
    For $\mu_{\text{mcf}}$ we calculate
    \begingroup
    \allowdisplaybreaks
    \begin{align*}
      \mu_{\text{mcf}} &= \int_{\T^n} - f(\nabla u) |\divergence v|^2 + \divergence v \left( f(\nabla u) \divergence v - Q \hess f(\nabla u) : \hess u \right) \\
                       &= \int_{\T^n} - f(\nabla u) |\divergence v|^2 + \divergence v \divergence \left( f(\nabla u)  v - Q \nabla f \right) \\
                       &= \int_{\T^n} - f(\nabla u) |\divergence v|^2 + \Dif v^T : \Dif \left( f(\nabla u) v - Q \nabla f \right) \\
                       &= \int_{\T^n} - f(\nabla u) |\divergence v|^2 \\ &\quad+ \int_{\T^{n}}\Dif v^T : \left( v \otimes \hess u \nabla f(\nabla u) + f(\nabla u) \Dif v - \nabla f \otimes \hess u v - Q \hess f \hess u \right) \\
                       &= \int_{\T^n} - f(\nabla u) |\divergence v|^2 + \Dif v^T : \left(f(\nabla u) \Dif v - Q \hess f \hess u \right) \\
                       &= \int_{\T^n} - f(\nabla u) |\divergence v|^2\\
                       &\quad+\int_{\T^{n}} \left( \hess u \left(\Id - v\otimes v\right) \right) : \left(\frac{f(\nabla u)}{Q^2} (\Id - v\otimes v) -\hess f  \right) \hess u. 
    \end{align*}
    For $\mu_{\text{pert}}$ we calculate
    \begin{align*}
      \mu_{\text{pert}} &= \int_{\T^n} -\divergence (\nabla f(\nabla u)) \lapl u + \nabla Q \cdot \hess f(\nabla u) \nabla Q \\
                        &=\int_{\T^n} -\Dif (\nabla f(\nabla u)) : \hess u + \nabla Q \cdot \hess f(\nabla u) \nabla Q \\
                        &=\int_{\T^n} - \hess u(\Id - v \otimes v) : \hess f(\nabla u) \hess u.
    \end{align*}
    \endgroup
    For $\sigma$ we calculate
    \begin{align*}
      \sigma &= \int_{\T^n} \nabla (f(\nabla u)) \cdot v = -\int_{\T^n} f(\nabla u) \divergence v.
    \end{align*}
    Inserting these calculations into \eqref{eq:lemma_dI} yields the result.
  \end{proof}
\end{lemma}

We will next explore for which choices of $f$ \Cref{lem:energy_ito_formula} yields a control on appropriate quantities. We therefore choose $f$ as a function of $\QQ$, which gives more geometric meaning to the estimates and is still sufficient to obtain the required estimates, see the remarks below.

\begin{lemma}
  \label{lem:energy_convex_integrand}
  Let $\varepsilon \ge 0$ and $u$ be a strong solution of \eqref{eq:smcf-viscous}.  Let $g \in C^2([1,\infty))$ be a non-negative, monotone increasing and convex function with bounded second order derivative and $g'(1) - g(1) \ge 0$ and
  \begin{align*}
    \mathcal{I}(t) = \int_{\T^n} g(\QQu), ~t \in I.
  \end{align*}
  For $q\in [1,2]$ we have that
  \begin{align*}
    \E &\mathcal{I}(t)^q + \varepsilon q\E\int_0^t  \mathcal{I}(s)^{q-1} \int_{\T^n} g''(\QQu[s]) |\hess u(s) \vvu[s]|^2 \ds \\
       &+ \varepsilon q \E\int_0^t \mathcal{I}(s)^{q-1} \int_{\T^n} \frac{g'(\QQu[s])}{\QQu[s]} \left(|\hess u(s)|^2 -  |\hess u(s) \vvu[s]|^2 \right) \ds \\
       &+ \frac{2q-q^2}{2} \E \int_0^t \mathcal{I}(s)^{q-1} \int_{\T^n} g(\QQu[s]) |\divvu[s]|^2 \ds \\
       &+ q \E \int_0^t \mathcal{I}(s)^{q-1} \int_{\T^n} \left(\frac{g'(\QQu[s])}{\QQu[s]} - \frac{g(\QQu[s])}{2\QQu[s]^2}\right) \\
       &\phantom{+ q \E \int_0^t \mathcal{I}(s)^{q-1}} ~\cdot \left(|\hess u|^2 - 2 |\hess u \vv|^2 + |\vv \cdot \hess u \vv|^2 \right)(s) \ds \\
       &+ q \E \int_0^t \mathcal{I}(s)^{q-1}\int_{\T^n} g''(\QQu[s]) \\
       &\phantom{+ q \E \int_0^t \mathcal{I}(s)^{q-1}}\cdot\left( |\hess u \vv|^2 - |\vv \cdot \hess u \vv|^2 \right)(s) \ds \\
       &\le \E \mathcal{I}(0)^q ~\forall t \in I.
  \end{align*}

  Furthermore, there is a constant $C > 0$ such that for $q \in [1,2)$
  \begin{align*}
    \E \sup_{t \in I} \mathcal{I}(t)^q \le \left(2 +  \frac{2C^2}{2q - q^2}\right) \E \mathcal{I}(0)^q.
  \end{align*}
\end{lemma}

\begin{remark}\
  \begin{enumerate}[(i)]
  \item Note that all terms on the left hand side in \Cref{lem:energy_convex_integrand} are non-negative. Especially non-negativity of $|\hess u|^2 - 2 |\hess u \vv |^2 + |\vv \cdot \hess u \vv|^2$ can be deduced from
  \begin{align*}
    |\hess u|^2 &- 2 |\hess u \vv |^2 + |\vv \cdot \hess u \vv|^2 \\
                &= Q^2 \Dif \left(\vv\right) : \Dif \left(\vv\right)^T\\
                &= \left(\Id - \vv\otimes \vv\right)\hess u : \hess u \left(\Id - \vv\otimes\vv\right)
  \end{align*}
  and \Cref{prop:sp_matrices}. The term $\Dif \left(\vv\right) : \Dif \left(\vv\right)^T$ is the squared norm of the second fundamental form of the graph of $u$. Hence \Cref{lem:energy_convex_integrand} yields a bound for this geometric quantity, see also \Cref{rem:geometric_energy} below.
  \item The condition $g'(1)-g(1) \ge 0$ in \Cref{lem:energy_convex_integrand} is not very restrictive since one can subtract a constant from $g$ and use the fact that $\dif \mathcal{I} = \dif \, (\mathcal{I} - \text{const})$.
  \end{enumerate}
\end{remark}

\begin{proof}[Proof of \Cref{lem:energy_convex_integrand}]
  Let $f(p) \coloneqq g(\QQ[p])$ for $p \in \R^n$. Then
  \begin{align*}
    \nabla f(p) &= g'(\QQ[p]) \vv[p], \\
    \hess f(p) &= g''(\QQ[p]) \vv[p] \otimes \vv[p] + g'(\QQ[p]) \frac{\Id - \vv[p] \otimes \vv[p]}{\QQ[p]} ~\forall p\in\R^n.
  \end{align*}
  Since $g''$ is bounded we infer that $g'$ grows at most linearly and therefore $\hess f$ is bounded. Furthermore, we calculate
  \begin{align}
    \label{eq:lemma_energy_convex_matrix}
    \begin{split}
      \hess f(p) &- \frac{f(p)}{2\QQ[p]^2} \left(\Id - \vv[p]\otimes\vv[p] \right) \\
      &= \left(\frac{g'(\QQ[p])}{\QQ[p]} - \frac{g(\QQ[p])}{2\QQ[p]^2}\right) \Id \\
      &\phantom{{}={}}+ \left(g''(\QQ[p]) - \frac{g'(\QQ[p])}{\QQ[p]} + \frac{g(\QQ[p])}{2\QQ[p]^2}\right)\vv[p]\otimes\vv[p].
      \end{split}
  \end{align}
  Note that
  \begin{align*}
    \od{}{\sigma} \left(g'(\sigma)\sigma - g(\sigma)\right) = g''(\sigma)\sigma \ge 0 ~\forall \sigma \in (1,\infty).
  \end{align*}
  Thus $\sigma \mapsto g'(\sigma)\sigma - g(\sigma)$ is an increasing function with $g'(1) - g(1) \ge 0$.
  
  Now the eigenvalues of \eqref{eq:lemma_energy_convex_matrix} are given by
  \begin{gather*}
    \frac{g'(\QQ[p])}{\QQ[p]} - \frac{g(\QQ[p])}{2\QQ[p]^2} \ge 0 \text{ and } \\
    \frac{g'(\QQ[p])}{\QQ[p]^{3}} - \frac{g(\QQ[p])}{2\QQ[p]^4} + g''(\QQ[p]) \frac{|p|^2}{\QQ[p]^2} \ge 0
  \end{gather*}
  which shows the non-negativity of \eqref{eq:lemma_energy_convex_matrix}.
  We will again use the notation $Q = \QQ$ and $v = \vv$.  We can apply \Cref{lem:energy_ito_formula} to $\mathcal{I}(t)$ and deduce
  \begin{align*}
    \dif \mathcal{I} &= \int_{\T^n} - \varepsilon \left(g''(Q) |\hess u v|^2 + \frac{g'(Q)}{Q} \left(|\hess u|^2 - |\hess u v|^2 \right) \right)\\
                     & +  \int_{\T^n} - \frac{1}{2} g(Q) |\divergence v|^2  - \int_{\T^n}  \hess u \left(\Id - v\otimes v\right)\\
                     &\qquad : \left(\left(\frac{g'(Q)}{Q} - \frac{g(Q)}{2Q^2}\right) (\Id - v \otimes v)  + g''(Q) v \otimes v\right) \hess u \dt \\
           & - \int_{\T^n} g(Q)\divergence v \dW.
  \end{align*}
  Because of the non-negativity of \eqref{eq:lemma_energy_convex_matrix} and \Cref{prop:sp_matrices}, $\mathcal{I}$ is a non-negative local supermartingale. We can apply Fatou's Lemma to get rid of the locality and deduce that
  \begin{align*}
    &\E  \int_{\T^n} g(Q(t)) \\
       &+ \varepsilon \E\int_0^t \int_{\T^n} \left(g''(Q(s)) |\hess u v|^2(s) + \frac{g'(Q(s))}{Q(s)} \left(|\hess u(s)|^2 - |\hess u v|^2(s) \right) \right) \ds \\
       &+ \frac{1}{2} \E \int_0^t \int_{\T^n} g(Q(s)) |\divergence v(s)|^2 \ds \\
       &+ \E \int_0^t \int_{\T^n} \left(\frac{g'(Q(s))}{Q(s)} - \frac{g(Q(s))}{2Q(s)^2}\right) \left(|\hess u|^2 - 2 |\hess u v|^2 + |v \cdot \hess u v|^2 \right)(s) \ds \\
       &+ \E \int_0^t \int_{\T^n} g''(Q(s)) \left( |\hess u v|^2 - |v \cdot \hess u v|^2 \right)(s) \ds \\
       &\le \E\int_{\T^n} g(Q(0)) ~\forall t \in I.
  \end{align*}
  
  Now, for $q \in [1,2]$ we want to use the It\^o formula for the function $x \mapsto |x|^q$. This function is not twice continuously differentiable for $q < 2$, so the classical It\^o formula does not apply directly.
  Nevertheless, we can first do the calculations for $\vartheta > 0$ and the function $x \mapsto \left(\vartheta + x\right)^q$ which is twice continuously differentiable on $[0,\infty)$ and then send $\vartheta \to 0$. We infer
  \begin{align*}
    \dif \mathcal{I}^q &= -\varepsilon q\mathcal{I}^{q-1} \int_{\T^n} \left(g''(Q) |\hess u v|^2 + \frac{g'(Q)}{Q} \left(|\hess u|^2 - |\hess u v|^2 \right) \right) \\
             &\phantom{{}={}} - \frac{q}{2} \mathcal{I}^{q-1} \int_{\T^n}  g(Q) |\divergence v|^2 \\
                       &\phantom{{}={}} -q \mathcal{I}^{q-1} \int_{\T^n}  \hess u \left(\Id - v\otimes v\right) \\
                       &\quad\quad\quad\quad\quad\quad: \left(\left(\frac{g'(Q)}{Q} - \frac{g(Q)}{2Q^2}\right) (\Id - v \otimes v)  + g''(Q) v \otimes v\right) \hess u \dt \\
             &\phantom{{}={}} + \frac{q(q-1)}{2} \mathcal{I}^{q-2} \left( \int_{\T^n} g(Q) \divergence v \right)^2 \dt \\
             &\phantom{{}={}} -q \mathcal{I}^{q-1} \int_{\T^n} g(Q)\divergence v \dW \\
             &\le -\varepsilon q\mathcal{I}^{q-1} \int_{\T^n} \left(g''(Q) |\hess u v|^2 + \frac{g'(Q)}{Q} \left(|\hess u|^2 - |\hess u v|^2 \right) \right) \\
             &\phantom{{}={}}+ \left(-\frac{q}{2}+\frac{q(q-1)}{2}\right) \mathcal{I}^{q-1} \int_{\T^n}  g(Q) |\divergence v|^2 \\
                       &\phantom{{}={}} -q \mathcal{I}^{q-1} \int_{\T^n}  \hess u \left(\Id - v\otimes v\right) \\
    &\quad\quad\quad\quad\quad\quad: \left(\left(\frac{g'(Q)}{Q} - \frac{g(Q)}{2Q^2}\right) (\Id - v \otimes v)  + g''(Q) v \otimes v\right) \hess u \dt \\
             &\phantom{{}={}} -q \mathcal{I}^{q-1} \int_{\T^n} g(Q)\divergence v \dW. \\
  \end{align*}
  As before, since the stochastic integral defines a local martingale and using Fatou's lemma, we get
  \begin{align*}
    &\E \left(\mathcal{I}(t)^q\right) \\
    &+ \varepsilon q \E\int_0^t  \mathcal{I}(s)^{q-1} \\
    &\quad\cdot\int_{\T^n} \left(g''(Q(s)) |\hess u v|^2(s) + \frac{g'(Q(s))}{Q(s)} \left(|\hess u(s)|^2 - |\hess u v|^2(s) \right) \right) \ds \\
    &+ \frac{2q-q^2}{2} \E \int_0^t \mathcal{I}(s)^{q-1} \int_{\T^n} g(Q(s)) |\divergence v(s)|^2 \ds \\
    &+ q \E \int_0^t \mathcal{I}(s)^{q-1} \int_{\T^n} \left(\frac{g'(Q(s))}{Q(s)} - \frac{g(Q(s))}{2Q(s)^2}\right) \\
    &\quad\qquad\quad\qquad\qquad\cdot\left(|\hess u|^2 - 2 |\hess u v|^2 + |v \cdot \hess u v|^2 \right)(s) \ds \\
    &+ q \E \int_0^t \mathcal{I}(s)^{q-1}\int_{\T^n} g''(Q(s)) \left( |\hess u v|^2 - |v \cdot \hess u v|^2 \right)(s) \ds \\
    &\le \E\left( \mathcal{I}(0)^q\right) ~\forall t \in I.
  \end{align*}
  For the stochastic integral we can apply the Burkholder-Davis-Gundy inequality
  \begin{align*}
    \E \sup_{t \in I} &\left[ \int_0^t \mathcal{I}(s)^{q-1} \int_{\T^n} g(Q(s)) \divergence v(s) \dW(s) \right] \\
                          &\le C \E \left[ \int_0^{\sup I} \left( \mathcal{I}(t)^{q-1} \int_{\T^n} g(Q(t)) \divergence v(t) \right)^2 \dt \right]^{\frac{1}{2}} \\
                          &\le C \E \left[ \int_0^{\sup I} \left(\mathcal{I}(t)^{2q-1} \int_{\T^n} g(Q(t)) |\divergence v(t)|^2\right) \dt \right]^{\frac{1}{2}} \\
                          &\le C \E \left[ \sup_{t \in I} \mathcal{I}^q(t) \int_0^{\sup I} \left( \mathcal{I}(t)^{q-1} \int_{\T^n} g(Q(t)) | \divergence v(t)|^2 \right) \dt \right]^{\frac{1}{2}} \\
                          &\le \frac{C\delta}{2} \E \sup_{t \in I} \mathcal{I}(t)^q + \frac{C}{2\delta} \E \int_0^{\sup I} \left( \mathcal{I}(t)^{q-1}\int_{\T^n} g(Q(t)) |\divergence v(t)|^2 \right) \dt \\
                          &\le \frac{C\delta}{2} \E \sup_{t \in I} \mathcal{I}(t)^q + \frac{C}{\delta(2q - q^2)} \E\left(\mathcal{I}(0)^q\right)
  \end{align*}
  for any $\delta > 0$.  Thus
  \begin{align*}
    \E \sup_{t \in I} \mathcal{I}(t)^q \le \E  \mathcal{I}(0)^q + \frac{C\delta}{2} \E \sup_{t \in I} \mathcal{I}(t)^q + \frac{C}{\delta(2q - q^2)} \E \mathcal{I}(0)^q.
  \end{align*}
  Now choose $\delta = \frac{1}{C}$ to infer
  \begin{align*}
    \E \sup_{t \in I} \mathcal{I}(t)^q \le \left(2 +  \frac{2C^2}{2q - q^2}\right) \E \mathcal{I}(0)^q.
  \end{align*}
\end{proof}

After \Cref{lem:energy_convex_integrand} has been established we can apply it to prove \Cref{prop:dirichlet} and \Cref{prop:maximum-principle}.

\begin{proof}[Proof of \Cref{prop:dirichlet}]
  In the situation of \Cref{lem:energy_convex_integrand} we choose $g(r) = r^2$, hence
  \begin{equation*}
    g(\QQ[p]) = 1 + |p|^2.
  \end{equation*}
  Note that $\dif \int_{\T^n} |\nabla u|^2 = \dif \int_{\T^n} g(\QQ)$.  Then by \Cref{lem:energy_convex_integrand} for $q=1$,
  \begin{align*}
    \E \| &\nabla u(t) \|_{L^2(\T^n)}^{2} + 2 \varepsilon \E\int_0^t \int_{\T^n} |\hess u(s)|^2 \ds \\
                                         &+ \frac{1}{2} \E \int_0^t \int_{\T^n} \QQ[\nabla u(s)]^2 |\divvu[s]|^2 \ds \\
                                         &+ \E \int_0^t \int_{\T^n} \left(\frac{3}{2}|\hess u|^2 - |\hess u \vv|^2 - \frac{1}{2} |\vv \cdot \hess u \vv |^2 \right)(s) \ds \\
                                         &\le \E \| \nabla u_0 \|_{L^2(\T^n)}^2 ~\forall t \in I.
  \end{align*}
  Furthermore there is a constant $C > 0$ such that for $q \in [1,2)$
  \begin{align*}
    \E \sup_{t \in I} \| \nabla u(t) \|_{L^2(\T^n)}^{2q} \le \left(2 +  \frac{2C^2}{2q - q^2}\right) \E  \| \nabla u_0 \|_{L^2(\T^n)}^{2q}.
  \end{align*}
  Now let $\esssup[\P] \| \nabla u \|_{L^\infty(I;L^\infty(\T^n))} = L < \infty$. Then we can estimate
  \begin{equation*}
    |\vv| = \frac{|\nabla u|}{\QQ} \le \frac{L}{\sqrt{1+L^2}} < 1
  \end{equation*}
  and
  \begin{align*}
    3|\hess u|^2 &- 2|\hess u \vv|^2 -  |\vv \cdot \hess u \vv |^2 \\
    &\ge \left(3 - 2 |\vv|^2 - |\vv|^4\right) |\hess u|^2 \\
    &\ge \frac{3(1+L^2)^2 - 2L^2(1+L^2) - L^4}{(1+L^2)^2} |\hess u|^2 \\
    &= \frac{3 + 4L^2}{(1+L^2)^2} |\hess u|^2.
  \end{align*}
  Hence
  \begin{align*}
    \E \| \nabla u(t) \|_{L^2(\T^n)}^{2} &+ \frac{3 + 4L^2}{2(1+L^2)^2} \E \int_0^t \int_{\T^n} |\hess u(s)|^2 \ds \\
                                         &\le \E \| \nabla u_0 \|_{L^2(\T^n)}^2 ~\forall t \in I.
  \end{align*}
\end{proof}

As the next step, we establish the maximum principle.
\begin{proof}[Proof of \Cref{prop:maximum-principle}]
  For $M > 0$ let $g_M \in C^2([1,\infty))$ be a modification of $\sigma \mapsto (\sigma - M)_+$ with the following properties:
  \begin{itemize}
  \item $g_M \ge 0$,
  \item $g_M$ monotone increasing and convex,
  \item $g_M''$ bounded,
  \item $g_M'(1) - g_M(1) \ge 0$ and
  \item $g_M(\sigma) > 0 \Leftrightarrow \sigma > M$.
  \end{itemize}
  For example, one could choose
  \begin{align*}
      g_M(\sigma) \coloneqq \left\{
    \begin{array}{ll}
      0 & \sigma \le M, \\
      (\sigma-M)^2 & \sigma \in (M, M+1), \\
      2 \sigma -  2M - 1 & \sigma \in [M+1, \infty).
    \end{array}
                           \right.
  \end{align*}
  From \Cref{lem:energy_convex_integrand} we deduce that
  \begin{align*}
    \E \int_{\T^n} g_M(\QQu) \le \E \int_{\T^n} g_M(\QQu[0]) ~\forall t \in I.
  \end{align*}
  Now, if $\| \nabla u_0 \|_{L^\infty(\T^n)} \le L$ $\P$-a.s. then we can conclude that $$\| \QQu[0] \|_{L^\infty(\T^n)} \le \sqrt{1+L^2}\ \P\text{-a.s.}$$ Hence
  \begin{align*}
    \E \int_{\T^n} g_{\sqrt{1+L^2}}(\QQu) = 0 ~\forall t \in I,
  \end{align*}
  which implies $\| \QQu \|_{L^\infty(\T^n)} \le \sqrt{1+L^2}$ $\P$-a.s. for all $t \in I$. Therefore
  \begin{align*}
    \| \nabla u \|_{L^\infty(I; L^\infty(\T^n))} \le L.
  \end{align*}
\end{proof}

\begin{remark}
  \label{rem:geometric_energy}
  We can also use \Cref{lem:energy_convex_integrand} for $g(r) = r$ to deduce bounds for the $q$-th moment of the area. In particular for $q = 1$ and $\varepsilon = 0$ we get
  \begin{align}
    \label{eq:geometric-energy}
    \begin{split}
    \E \int_{\T^n} \QQu[t] &+  \frac{1}{2} \E \int_0^t  \int_{\T^n} \QQu[s] \left| \divvu[s] \right|^2 \ds \\
                           &+ \frac{1}{2} \E \int_0^t  \int_{\T^n} \QQu[s]  \Dif \vvu[s] : \Dif \vvu[s]^T \ds \\
                           &\le \E \int_{\T^n} \QQu[0].
                         \end{split}
  \end{align}
  In geometrical terms \eqref{eq:geometric-energy} becomes
  \begin{align*}
    \E \mathcal{H}^n\left(\Gamma(t) \right) + \frac{1}{2} \E \int_0^t \int_{\Gamma(t)} \left( H^2(s) + |A(s)|^2 \right) \, \dif\mathcal{H}^n \ds \le \E \mathcal{H}^n\left(\Gamma(0)\right),
  \end{align*}
  where $\Gamma(t) = \graph u(t)$, $H(t)$ is the mean curvature and $|A(t)|$ is the length of the second fundamental form of $\Gamma(t)$ for $t \in I$. Compare this with the deterministic MCF, where for a solution $(\Gamma(t))_{t \ge 0}$ the natural energy identity is 
  \begin{align*}
    \mathcal{H}^n\left(\Gamma(t)\right) + \int_0^t \int_{\Gamma(t)} H^2(s) \, \dif\mathcal{H}^n \ds = \mathcal{H}^n\left(\Gamma(0)\right).
  \end{align*}
  However, we will not use this estimate since an $L^\infty$ bound for the gradient and an $L^2$ bound for the Hessian are available via \Cref{prop:maximum-principle} and \Cref{prop:dirichlet}.
\end{remark}

\section{Vanishing viscosity limit}
  
With the above uniform estimates at hand, we are in position to pass to the limit as $\varepsilon\to 0$ and establish the existence of a martingale solution to the stochastic mean curvature flow  \eqref{eq:smcf}.
  
  \label{sec:vanishing-viscosity}
  \begin{proof}[Proof of \Cref{thm:main}]
    From \Cref{thm:existence_viscous} we deduce that for $\varepsilon > 0$ we can find a martingale solutions $u^\varepsilon$ of \eqref{eq:smcf-viscous} with initial data $\Lambda$. Since the solutions $u^\varepsilon$ are constructed with \cite[Theorem 2]{Jakubowski1997} we can fix one probability space $(\Omega, \F, \P)$ = $([0,1], \mathcal{B}([0,1]), \mathcal{L})$ such that for each $\varepsilon > 0$ we can find
    \begin{itemize}
    \item a normal filtration $(\F^\varepsilon_t)_{t \in [0,\infty)}$,
    \item a real-valued $(\F^\varepsilon_t)$-Wiener process $W^\varepsilon$ and
    \item a $(\F^\varepsilon_t)$-predictable process $u^\varepsilon$ with $u^\varepsilon \in L^2(\Omega; L^2(0,T;H^2(\T^n)))$ for all $T \in [0,\infty)$
    \end{itemize}
    such that
    \begin{align}
      \label{eq:epsilon-du}
      \begin{split}
      u^\varepsilon(t) - u^\varepsilon(0) &= \int_0^t \varepsilon \lapl u^\varepsilon(s) + \QQ[\nabla u^\varepsilon(s)] \divv[\nabla u^\varepsilon(s)] \ds \\
      &\phantom{{}={}} + \int_0^t \QQ[\nabla u^\varepsilon(s)] \circ \dW^\varepsilon_s \text{ in } L^2(\T^n) ~\forall t \in [0,\infty)
      \end{split}
    \end{align}
    and $\P \circ (u^\varepsilon(0))^{-1} = \Lambda$.
    Because of the assumption on the support of $\Lambda$ we have $\| \nabla u^\varepsilon(0) \|_{L^\infty(\T^n)} \le L$ $\P$-a.s.
    From \Cref{prop:maximum-principle} we deduce
    \begin{align*}
      \| \nabla u^\varepsilon \|_{L^\infty(0,\infty;L^\infty(\T^n))} \le L ~\P\text{-a.s.}
    \end{align*}
    and from \Cref{prop:dirichlet} we deduce for all $q \in [1,2)$
    \begin{align}
      \label{eq:epsilon-bounds}
      \begin{split}
      \| \nabla u^\varepsilon \|_{L^{2q}(\Omega; C([0,\infty);L^2(\T^n)))}^{2q} &\le \left(2 + \frac{2C^2}{2q - q^2} \right)  \E \| \nabla u^\varepsilon(0) \|_{L^2(\T^n)}^{2q} \le C_{q,L} \text{ and } \\
      \| \hess u^\varepsilon \|_{L^2(\Omega; L^2(0,\infty;L^2(\T^n)))} &\le C_L \| u^\varepsilon(0) \|_{L^2(\Omega;H^1(\T^n))}.
      \end{split}
    \end{align}
    Using \Cref{cor:variational-sobolev-ito} we infer that 
    \begin{align*}
      \dif \| u^\varepsilon \|_{L^2(\T^n)}^2 &= \int_{\T^n} 2 u^\varepsilon \\
                                             &\qquad\cdot\left( \varepsilon \lapl u^\varepsilon + \QQ[\nabla u^\varepsilon] \divv[\nabla u^\varepsilon] + \frac{1}{2} \vv[\nabla u^\varepsilon] \cdot \hess u^\varepsilon \vv[\nabla u^\varepsilon] \right) \dt \\
                                             &\phantom{{}={}} + \int_{\T^n} \QQ[\nabla u^\varepsilon]^2 \dt  + 2 \int_{\T^n} u^\varepsilon \QQ[\nabla u^\varepsilon] \dW \\
                                             &\le C \| u^\varepsilon \|_{L^2(\T^n)}^2 + C \| \hess u^\varepsilon \|_{L^2(\T^n)}^2 \dt + 2 \int_{\T^n} u^\varepsilon \QQ[\nabla u^\varepsilon] \dW,
    \end{align*}
    with a constant $C$ that does not depend on $\varepsilon$. We can estimate the supremum with the Burkholder-Davis-Gundy inequality and get
    \begin{align*}
      &\E \sup_{s \in [0,t]} \| u^\varepsilon(s) \|_{L^2(\T^n)}^2 - \E \| u^\varepsilon(0)\|_{L^2(\T^n)}^2 \\
      &\le C \E \int_0^t \| u^\varepsilon(s) \|_{L^2(\T^n)}^2 \ds + C \E \int_0^t \| \hess u^\varepsilon(s) \|_{L^2(\T^n)}^2 \\
                                                                                                          &\phantom{{}={}} + C \E \left[ \int_0^t \| u^\varepsilon(s) \|_{L^2(\T^n)}^2 \| \QQ[\nabla u^\varepsilon(s) ] \|_{L^2(\T^n)}^2 \ds \right]^{\frac{1}{2}} \\
                                                                                                          &\le C \E \int_0^t \| u^\varepsilon(s) \|_{L^2(\T^n)}^2 \ds + C \E \int_0^t \| \hess u^\varepsilon(s) \|_{L^2(\T^n)}^2 \\
                                                                                                          &\phantom{{}={}} + \frac{1}{2} \E \sup_{s \in [0,t]}  \| u^\varepsilon(s) \|_{L^2(\T^n)}^2 + C \E \int_0^t \| \QQ[\nabla u^\varepsilon(s) ] \|_{L^2(\T^n)}^2 \ds.
    \end{align*}
    From \eqref{eq:epsilon-bounds} we infer for all $T > 0$ and $t \in [0,T]$ that
    \begin{align*}
      \E \sup_{s \in [0,t]} \| u^\varepsilon(s) \|_{L^2(\T^n)}^2 &\le C \E \int_0^t \| u^\varepsilon(s) \|_{L^2(\T^n)}^2 \ds + C,
    \end{align*}
    with a constant  that only depend on the initial condition $\Lambda$ and $T$.
    Using the Gronwall lemma we conclude that there is a constant $C$ which only depends on $\Lambda$ and $T$ such that
    \begin{align*}
      \E \sup_{t \in [0,T]} \| u^\varepsilon(s) \|_{L^2(\T^n)}^2 \le C.
    \end{align*}

    Because of \eqref{eq:epsilon-bounds} we know that the deterministic integral in \eqref{eq:epsilon-du} is uniformly bounded in $L^2(\Omega; C^{0,\frac{1}{2}}([0,T]; L^2(\T^n)))$.
    With the factorization method \cite[Theorem 1.1]{Seidler1993} and \eqref{eq:epsilon-bounds} we infer that there is a $\lambda > 0$ such that the stochastic integral in \eqref{eq:epsilon-du} is uniformly bounded in $L^2(\Omega; C^{0, \lambda}([0,T]; L^2(\T^n)))$. Hence, for some $\lambda \in (0, \frac{1}{2})$
    \begin{align*}
      \E \| u^\varepsilon \|_{C^{0,\lambda}([0,T]; L^2(\T^n))}^2 \le C_{\Lambda, T}
    \end{align*}
    uniformly in $\varepsilon$.
    We conclude that $(u^\varepsilon)_{\varepsilon > 0}$ is uniformly bounded in
    \begin{align*}
      L^2\left(\Omega; L^2(0,T;H^2(\T^n)) \cap C([0,T];H^1(\T^n)) \cap C^{0,\lambda}([0,T]; L^2(\T^n))\right)
    \end{align*}
    and $(W^\varepsilon)_{\varepsilon > 0}$ is uniformly bounded in $L^2\left(\Omega; C^{0,\lambda}([0,T];\R) \right)$.

    In the remaining part of the proof we will show that these bounds imply the existence of a convergent subsequence in a weak sense and we will identify the limit with a solution of \eqref{eq:smcf}. We will follow the same strategy in the proof of \Cref{thm:variational-existence}, where we pass from the finite-dimensional approximations to a solution. Since the line of arguments is very similar but slightly more involved in the case of \Cref{thm:variational-existence} we give a detailed proof only for the latter Theorem and here just comment on the main ideas and on differences between both proofs.

    Using the compactness of the embeddings, the joint laws of $(u_\varepsilon, W_\varepsilon)$ are tight in
    \begin{align*}
      \mathcal{X}_u^T \times \mathcal{X}_W^T
    \end{align*}
    with
    \begin{align*}
      \mathcal{X}_u^T &\coloneqq C([0,T]; (H^1(\T^n), w)) \cap L^2(0,T;H^1(\T^n)) \cap \left(L^2(0,T; H^2(\T^n)), w \right), \\
      \mathcal{X}_W^T &\coloneqq C([0,T];\R).
    \end{align*}
    Since $T > 0$ is arbitrary this also implies the tightness in $\mathcal{X}_u \times \mathcal{X}_W$ with
    \begin{align*}
      \mathcal{X}_u &\coloneqq C_{\text{loc}}([0,\infty); (H^1(\T^n), w)) \cap L^2_{\text{loc}}(0,\infty;H^1(\T^n)) \\
                    &\qquad\cap \left(L^2_{\text{loc}}(0,\infty; H^2(\T^n)), w \right), \\
      \mathcal{X}_W &\coloneqq C_{\text{loc}}([0,\infty);\R).
    \end{align*}
    Now we can argue via the Jakubowski-Skorokhod representation theorem for tight sequences in nonmetric spaces \cite[Theorem 2]{Jakubowski1997} to deduce the existence of a subsequence $\varepsilon_k \searrow 0$, a probability space $(\tilde{\Omega}, \tilde{\F}, \tilde{\P})$ and $\mathcal{X}_u \times \mathcal{X}_W$-valued random variables $(\tilde{u}^k, \tilde{W}^k)$ for $k \in \N$ and $(\tilde{u}, \tilde{W})$ such that $\tilde{u}^k \to \tilde{u}$ a.s. in $\mathcal{X}_u$, $\tilde{W}^k \to \tilde{W}$ a.s. in $\mathcal{X}_W$ and the joint laws of $(\tilde{u}^k, \tilde{W}^k)$ agree with the joint laws of $(u^{\varepsilon_k}, W^{\varepsilon_k})$ for $k \in \N$.    

    Let
    \begin{align*}
      \tilde{\F}^k_t &\coloneqq \bigcap_{s > t} \sigma\left(\tilde{u}^k|_{[0,s]}, \tilde{W}^k|_{[0,s]}, \{ A \in \tilde{\F} \mid \tilde{\P}(A) = 0 \} \right), ~t\in[0,\infty), \,k\in\N\\
      \tilde{\F}_t &\coloneqq \bigcap_{s > t} \sigma\left(\tilde{u}|_{[0,s]}, \tilde{W}|_{[0,s]}, \{ A \in \tilde{\F} \mid \tilde{\P}(A) = 0 \} \right), ~t\in[0,\infty). 
    \end{align*}
    One can prove that $\tilde{W}^k$ is a real-valued $(\tilde{\F}^k_t)_t$-Wiener process and $\tilde{u}^k$ is a solution of \eqref{eq:smcf-viscous} for $\varepsilon_k$ and the Wiener process $\tilde{W}^k$.

    With the a.s. convergences in $\mathcal{X}_W$ resp. $\mathcal{X}_u$ and the uniform bounds derived before one can pass to the limit in the equations and infer that $\tilde{W}$ is a real-valued $(\tilde{F}_t)_t$-Wiener process and $\tilde{u}$ is a solution of \eqref{eq:smcf}. In opposite to the proof of \Cref{thm:variational-existence} the operator in the deterministic part of the equation changes, but the convergence $u^\varepsilon \rightharpoonup u$ in $H^2(\T^n))$ implies
    \begin{align*}
      \varepsilon \Delta u^\varepsilon &+ \QQ[\nabla u^\varepsilon]\divv[\nabla u^\varepsilon] + \frac{1}{2} \vv[\nabla u^\varepsilon] \cdot \hess u^\varepsilon \vv[\nabla u^\varepsilon] \\
      &\rightharpoonup \QQ[\nabla u]\divv[\nabla u] + \frac{1}{2} \vv[\nabla u] \cdot \hess u \vv[\nabla u] \text{ in } L^2(\T^n),
    \end{align*}
    which is enough to pass to the limit in the equation.
    Because of the uniform bounds of $(u^\varepsilon)$ in $L^2(\Omega; L^2(0,T;H^2(\T^n)))$ for all $T > 0$ we know that the limit $\tilde{u}$ is already a martingale solution.
  \end{proof}

  \begin{remark}
    \label{rem:regularity-of-solution}
    Note that under the assumptions of \Cref{thm:main} for a martingale solution $u$ of \eqref{eq:smcf} we have $u \in C([0,t]; H^1(\T^n))$ and $\nabla u \in L^\infty(0,\infty; L^\infty(\T^n))$ $\P$-a.s. for all $t \in I$.
    
    For a function $w \in L^1(\T^n)$ with $\nabla w \in L^\infty(\T^n)$ we have $w \in W^{1,\infty}(\T^n)$ with
  \begin{align*}
    \| w \|_{W^{1,\infty}(\T^n)} \le C\left(\| w \|_{L^1(\T^n)} + \| \nabla w \|_{L^\infty(\T^n)}\right).
  \end{align*}
  Hence, $u \in L^\infty(0,t;W^{1,\infty}(\T^n))$ $\P$-a.s. for all $t \in I$.

  From the proof of \Cref{thm:main} we deduce that $u \in L^2(\Omega; C^{0,\lambda}([0,t]; L^2(\T^n)))$ for all $t \in I$ and some $\lambda > 0$. In combination with the previous result and \cite[Theorem 5]{Simon1987} this yields $u \in C([0,t]; C(\T^n))$ $\P$-a.s. for all $t \in I$. Using sharper interpolation results one can prove that the solution is pathwise Hölder continuous in space and time.
  \end{remark}

  \section{Large-time behavior}
  
  In this section, we study the large-time behavior of solutions to \eqref{eq:smcf}.
  
  \label{sec:large-time-behavior}
  \begin{proof}[Proof of \Cref{thm:large-time}]
   The uniform estimates in \Cref{prop:dirichlet} and \Cref{prop:maximum-principle} imply that
    \begin{align*}
      \hess u \in L^2(\Omega; L^2(0,\infty; L^2(\T^n)))
    \end{align*}
    and
    \begin{align*}
      \| \nabla u \|_{L^\infty(0,\infty; L^\infty(\T^n))} \le L ~\text{a.s.}
    \end{align*}

    For the convergence as $T \to \infty$ we note that by \Cref{cor:variational-sobolev-ito}
      \begin{align*}
        \dif \int_{\T^n} \left(u - W\right) &= \int_{\T^n} \QQ \divv \dt + \int_{\T^n} \left(\QQ-1\right) \circ \dW \\
                           &= -\frac{1}{2} \int_{\T^n} \vv \cdot \hess u \vv \dt + \int_{\T^n} \left(\QQ-1\right) \dW.
      \end{align*}
      Let
      \begin{align*}
        \alpha \coloneqq \frac{1}{|\T^n|} \Bigg( \int_{\T^n} &u_0 - \frac{1}{2} \int_0^\infty \int_{\T^n} \vvu \cdot \hess u(t) \vvu \dt  \\
        &+ \int_0^\infty \int_{\T^n} \left(\QQu - 1 \right) \dW(t) \Bigg).
      \end{align*}
      To bound the drift we estimate
      \begin{align*}
        \left| \int_{\T^n} \vv \cdot \hess u \vv \right| \le \int_{\T^n} \left| \hess u \right| \left| \nabla u \right| \le \| \nabla u \|_{H^1(\T^n)}^2 \le C \| \hess u \|_{L^2(\T^n)}^2,
      \end{align*}
      where we have used $|\vv[p]| \le \min \{ |p|, 1 \}$ and a Poincar\'e inequality, and infer
      \begin{align*}
        \E \left| \int_0^\infty \int_{\T^n} \vvu \cdot \hess u(t) \vvu \dt \right| \le C \| \hess u \|_{L^2(\Omega; L^2(0,\infty; L^2(\T^n)))}^2 < \infty.
      \end{align*}
      Furthermore we have $\QQ[p] - 1 \le |p|$ and therefore for the martingale part of $\alpha$ the bound
      \begin{align*}
        \E \left| \int_0^\infty \int_{\T^n} \left(\QQu - 1 \right) \dW(t) \right|^2 &\le \E \int_0^\infty \| \nabla u(t) \|_{H^1(\T^n)}^2 \dt  \\
        &\le C \E \int_0^\infty \| \hess u(t) \|_{L^2(\T^n)}^2 \dt < \infty,
      \end{align*}
      hence $\alpha \in L^1(\Omega)$ is a well-defined random variable.

      \begin{sloppypar}We find a sequence $(t_k)_{k\in\N}$ of increasing times $t_k \to \infty$ such that $\E \| \hess u(t_k) \|_{L^2(\T^n)}^2 \to 0$ for $k \to \infty$.\end{sloppypar}
      Now, we apply a Poincar\'e inequality to obtain
      \begin{align}
        \label{eq:large-time-estimate}
        \| u(t) - W(t) - \alpha \|_{H^1(\T^n)} \le C \left(  \| \nabla u(t) \|_{L^2(\T^n)} + \left| \int_{\T^n} \left( u(t) - W(t) - \alpha \right) \right| \right).
      \end{align}
      From \Cref{prop:dirichlet} we infer that $\| \nabla u \|_{L^2(\T^n)}^2$ is a non-negative supermartingale and
      \begin{align*}
        \E \sup_{t \ge T} \| \nabla u(t) \|_{L^2(\T^n)}^2 \le C \E \| \nabla u(T) \|_{L^2(\T^n)}^2 \le C \E \| \nabla u(t_k) \|_{L^2(\T^n)}^2
      \end{align*}
      for $t_k < T$. Hence
      \begin{align*}
        \lim_{T \to \infty} \left( \E \sup_{t \ge T} \| \nabla u(t) \|_{L^2(\T^n)} \right)^2 &\le \lim_{T \to \infty} \E \sup_{t \ge T} \| \nabla u(t) \|_{L^2(\T^n)}^2 \\
                                                                                             &\le C \lim_{k \to \infty} \E \| \nabla u(t_k) \|_{L^2(\T^n)}^2 \\
                                                                                             &\le C \lim_{k \to \infty} \E \| \hess u(t_k) \|_{L^2(\T^n)}^2 = 0.
      \end{align*}
      For the second term in \eqref{eq:large-time-estimate} we have with the Burkholder-Davis-Gundy inequality and the estimates from above
      \begin{align*}
        \E& \sup_{t \ge T} \left| \int_{\T^n} \left( u(t) - W(t) - \alpha \right) \right| \\
                          &\le \E \sup_{t \ge T} \left| -\frac{1}{2} \int_t^\infty \int_{\T^n} \vvu[s] \cdot \hess u(s) \vvu[s] \ds\right| \\&\phantom{{}={}} + \E \sup_{t \ge T} \left| \int_t^\infty \int_{\T^n} \left(\QQu[s] - 1 \right) \dW(s) \right| \\
          &\le C \E \int_T^\infty \| \hess u(t) \|_{L^2(\T^n)}^2 \dt  
        + C \left( \E \int_T^\infty \| \nabla u(t) \|_{L^2(\T^n)}^2 \dt \right)^{\frac{1}{2}} \\
                                                                            &\to 0 \text{ for } T \to \infty.
      \end{align*}

    \end{proof}

  From \Cref{thm:large-time} we can deduce the next corollary which extends the one-dimensional result from \cite[Theorem 4.2]{Es-Sarhir2012} to higher dimensions. Furthermore it improves the convergence in distribution in $C_{\text{loc}}([0,\infty); L^2(\T^n))$ to convergence in $L^1(\Omega; C_b([0,\infty), H^1(\T^n)))$.
  \begin{corollary}
    Let $u$ be a strong solution of \eqref{eq:smcf} with $$\esssup[\P] \| \nabla u_0 \|_{L^\infty(\T^n)} < \infty.$$ Then for $T \to \infty$ we have
    \begin{align*}
      (u(T+t) - u(T))_{t \ge 0} - (W(T+t) - W(T))_{t \ge 0} \to 0 
    \end{align*}
    in $L^1(\Omega; C_b([0,\infty); H^1(\T^n)))$.

    \begin{proof}
      We estimate
      \begin{align*}
        \E \sup_{t \ge 0}& \| u(T+t) - u(T) - (W(T+t)-W(T)) \|_{H^1(\T^n)} \\
        &\le 2\E  \sup_{t \ge T} \| u(t) - W(t) - \alpha \|_{H^1(\T^n)} \to 0.
      \end{align*}
    \end{proof}
  \end{corollary}

  \appendix

  % !TEX root = main.tex
\section{Variational SPDE under a compactness assumption}
\label{sec:variational-spdes}

In this section we will consider infinite-dimensional stochastic differential equations with a variational structure. In \Cref{sec:variational-ito} we will prove an It\^o formula for this kind of equation, which will be used in \Cref{sec:variational-existence} to show existence for variational SPDEs.
During the whole section we will work with the following assumptions.

\begin{assumptions}
  \label{assumptions:variational}
  Let $V$ and $H$ be separable Hilbert spaces with $V \subset H \simeq H' \subset V'$ and $V$ densely and compactly embedded in $H$. Furthermore we will consider another separable Hilbert space $U$, which will be the space where a Wiener process is defined. For notational convenience we will restrict the presentation to the case of infinite-dimensional spaces, although finite-dimensional spaces could be treated as well.

  Then we can find an orthonormal basis $(e^k)_{k \in \N}$ of $H$ which is an orthogonal basis of $V$ and we will use the abbreviation $\lambda_k \coloneqq \| e^k \|_V^2$ for $k \in \N$. We will assume that the $(e^k)_k$ are arranged such that $(\lambda_k)_k$ is a non-decreasing sequence.
  Furthermore we will denote by $(g_l)_{l\in\N}$ an orthonormal basis of $U$.

  If not otherwise specified then a cylindrical Wiener process $W$ on $U$ with respect to a filtration $(\F_t)_t$ will always be assumed to have the representation
  \begin{align*}
    W = \sum_{l\in\N} g_l \beta_l
  \end{align*}
  with $(\beta_l)_{l\in\N}$ mutually independent real-valued $(\F_t)$-Brownian motions.
\end{assumptions}

\begin{remark}[Existence of $(e^k)_{k \in \N}$]
  Let $J_V : V \to V'$ be the identification of $V$ with its dual space $V'$ via
  \begin{align*}
    \sprod[V',V]{J_Vx}{y} \coloneqq \sprod[V]{x}{y},~x,y \in V.
  \end{align*}
  The restriction of the inverse of $J_V$ to $H$ together with the embedding of $V$ into $H$ gives $J_V^{-1}|_{H} : H \to H$ which is compact and self-adjoint. Hence we find an orthonormal basis $(e^k)_{k\in\N}$ of $H$ of eigenvectors of $J_V^{-1}$ with corresponding eigenvalues $\left(\frac{1}{\lambda_k}\right)_{k \in \N}$. For $k \in \N$ we have that $e^k = \lambda_k J^{-1}e^k \in V$ and
  \begin{align*}
    \sprod[V]{e^k}{e^l} = \sprod[V',V]{J_Ve^k}{e^l} = \lambda_k \sprod[H]{e^k}{e^l} = \lambda_k \delta_{k,l} ~\forall k,l \in \N.
  \end{align*}
\end{remark}

\subsection{It\^o formula}
\label{sec:variational-ito}
The following result is a generalization of a well-known It\^o formula for variational SPDEs, cf.~\cite[Theorem 4.2.5]{Prevot2007} and \cite[II.II.§4]{Pardoux1975}.

\begin{proposition}[It\^o formula and continuity]
  \label{prop:variational-ito}
Assume that $T>0$, $\stochbasis$ is a stochastic basis with a normal filtration and $W$ a cylindrical Wiener process on $U$. Furthermore let $u_0 \in L^2(\Omega; H)$ be $\F_0$-measurable and $u$, $v$, $B$ be predictable processes with values in $V$, $V'$ and $L_2(U;H)$ , respectively, such that
  \begin{align*}
    u &\in L^2(\Omega; L^2(0,T;V)), \\
    v &\in L^2(\Omega; L^2(0,T;V')),\\
    B &\in L^2(\Omega; L^2(0,T;L_2(U;H))),
  \end{align*}
  and
  \begin{align}
    \label{eq:ito-du}
    u(t) - u_0 = \int_0^t v(s) \ds + \int_0^t B(s) \dW(s)~ \text{ in } V' ~ \P\text{-a.s.} ~\forall t \in [0,T].
  \end{align}
  Then $u$ has a version with continuous paths in $H$ and for this version it holds that $u \in L^2(\Omega; C([0,T]; H))$ with
  \begin{align*}
    \| u(t) \|_H^2 - \| u_0 \|_H^2 &= \int_0^t 2 \sprod[V',V]{v(s)}{u(s)} + \| B(s) \|_{L_2(U;H)}^2 \ds \\
                                   &\phantom{{}={}} + 2 \int_0^t \sprod[H]{u(s)}{B(s) \dW(s)} ~\forall t \in [0,T].
  \end{align*}
  Furthermore, if $F \in C^1(H)$ and the second G\^ateaux derivative $\hess F : H \to L(H)$ exists with
  \begin{itemize}
  \item $F$, $\Dif F$ and $\hess F$ bounded on bounded subsets of $H$,
  \item $\hess F : H \to L(H)$ continuous from the strong topology on $H$ to the weak-$\ast$ topology on $L(H) = \left(L_1(H)\right)^\ast$ and
  \item $(\Dif F)|_{V} : V \to V$ continuous from the strong topology on $V$ to the weak topology on $V$ and growing only linearly
  \begin{align*}
    \| \Dif F(x) \|_V \le C\left(1 + \| x\|_V\right)  ~\forall x \in V,
  \end{align*}
  \end{itemize}
  then
  \begin{align*}
    F(u(t)) - F(u_0) &= \int_0^t \sprod[V',V]{v(s)}{\Dif F(u(s))} \\
    &\phantom{{}={}} +  \frac{1}{2} \tr \left[ \hess F(u(s)) B(s) (B(s))^\ast \right] \ds  \\
                     &\phantom{{}={}} + \int_0^t \sprod[H]{\Dif F(u(s))}{B(s) \dW(s)} ~\P\text{-a.s.} ~\forall t \in [0,T].
  \end{align*}
\end{proposition}

Note that this generalization is similar to the result presented \cite[II.II.§4]{Pardoux1975}, where it was proven under slightly different assumptions, e.g.~ the embedding $V \subset H$ is not assumed to be compact and $V$ is only a separable Banach space, which is uniformly smooth and convex, but $F$ is assumed to be twice Fr\'echet differentiable.
By analyzing the proof in \cite[II.II.§4]{Pardoux1975} one can see that the assumptions on the differentiability of $F$ can be relaxed. It is sufficient to assume the G\^ateaux differentiability of $\Dif F : H \to H$ and a weak continuity of $\hess F : H \to L(H)$, which ensures that the restriction $F|_V : V \to \R$ is twice Fr\'echet differentiable.

Similar results are also stated in \cite[I.3.2]{Pardoux1979} and \cite[Lemma 2.3.5]{Pardoux2007} without proof and in \cite[I.1.2]{Pardoux1979} for a finite-dimensional noise. In \cite[I.§1 Theorem 1.3]{Viot1976} the result is proven by a finite-dimensional approximation.

For the readers convenience we will discuss the main steps of a different proof here, which makes use of the stronger assumptions compared to \cite{Pardoux1975} and adapts the proof of \cite[Proposition A.1]{Debussche2016}, where $H$ and $V$ are assumed to be Sobolev spaces. In \eqref{eq:variational-smoothing} below, we define a smoothing operator which in fact is the semigroup generated by $-J_V$ with $J_V : V \to V'$ given by the canonical identification of the Hilbert space $V$ with its dual space. If $V = H^1_0(\T^n)$, $H = L^2(\T^n)$, then this is the classical heat semigroup.
    
  \begin{proof}[Proof of \Cref{prop:variational-ito}]
      \emph{Step 1: Smoothing the solution.}\\
      For $\varepsilon > 0$ we will consider the smoothing operator
      \begin{equation}
        \label{eq:variational-smoothing}
        \begin{gathered}
        \rho_\varepsilon : V' \to V ,\qquad
        \rho_\varepsilon(v') \coloneqq \sum_{k \in \N} \exp(-\varepsilon \lambda_k) \sprod[V',V]{v'}{e^k} e^k.
        \end{gathered}
      \end{equation}
      It is easy to verify that with $\varepsilon \to 0$
      \begin{align*}
        \rho_\varepsilon(x) &\to x \text{ in } V' ~\forall x \in V', \\
        \rho_\varepsilon(x) &\to x \text{ in } H ~\forall x \in H \text{ and}\\
        \rho_\varepsilon(x) &\to x \text{ in } V ~\forall x \in V.
      \end{align*}
      Since $\rho_\varepsilon \in L(V'; V)$ we have for all $t \in [0,T]$ that
      \begin{align*}
        \rho_\varepsilon(u(t)) - \rho_\varepsilon(u_0) = \int_0^t \rho_\varepsilon(v(s)) \ds + \sum_{l\in\N} \int_0^t \rho_\varepsilon(B(s) g_l) \dif \beta_l(s) ~\P\text{-a.s. in } V.
      \end{align*}
      In the following we will abbreviate $u_{0,\varepsilon} \coloneqq \rho_\varepsilon(u_0)$, $u_\varepsilon(t) \coloneqq \rho_\varepsilon(u(t))$, $v_\varepsilon(t) \coloneqq \rho_\varepsilon(v(t))$ and $B_\varepsilon(t)g_l \coloneqq \rho_\varepsilon(B(t) g_l)$.

      \emph{Step 2: $u$ takes values in $H$}\\
      We will apply the It\^o formula \cite[Theorem 4.32]{DaPrato2014} for the function $ \| \cdot \|_H^2$. This gives
      \begin{align*}
        \| u_\varepsilon(t) \|_H^2 - \| u_{0,\varepsilon} \|_H^2 &= \int_0^t 2 \sprod[H]{u_\varepsilon(s)}{v_\varepsilon(s)} + \| B_\varepsilon(s) \|_{L_2(U;H)}^2 \ds \\
                                                                 &\phantom{{}={}}+2  \int_0^t \sprod[H]{u_\varepsilon(s)}{B_\varepsilon(s) \dW(s)}.
      \end{align*}
      Using the Burkholder-Davis-Gundy inequality and the assumptions on $u$, $v$ and $B$ one can conclude that $\E \sup_{t \in [0,T]} \| u_\varepsilon(t) \|_H^2$ is uniformly bounded in $\varepsilon$. With Fatou's Lemma we infer that
      \begin{align}
        \label{eq:ito-sup-H}
        \begin{split}
          \E \sup_{t \in [0,T]} \| u(t) \|_H^2 &=  \E \sup_{t \in [0,T]} \lim_{\varepsilon \to 0} \| u_\varepsilon(t) \|_H^2  \\
          &\le \E \liminf_{\varepsilon \to 0} \sup_{t \in [0,T]} \| u_\varepsilon(t) \|_H^2 \\
          &\le \liminf_{\varepsilon \to 0} \E \sup_{t \in [0,T]} \| u_\varepsilon(t) \|_H^2 < \infty.
        \end{split}
      \end{align}

      \emph{Step 3: Proving the It\^o formula.}\\
      Now, in \Cref{lem:variational-continuity-functional} it is verified that $F|_V \in C^2(V)$ with $F|_V$, $\Dif (F|_V)$ and $\hess (F|_V)$ uniformly continuous on bounded subsets of $V$.
      We apply the It\^o formula \cite[Theorem 4.32]{DaPrato2014} to conclude that
      \begin{align}
        \label{eq:ito-approx}
        \begin{split}
          F(u_\varepsilon(t)) - F(u_{0,\varepsilon}) &= \int_0^t \sprod[V',V]{v_\varepsilon(s)}{\Dif F(u_\varepsilon(s))} \\
          &\phantom{{}={}} + \frac{1}{2}\int_0^t \tr \left[ \hess F(u_\varepsilon(s)) B_\varepsilon(s) (B_\varepsilon(s))^\ast \right] \ds  \\
          &\phantom{{}={}} + \int_0^t \sprod[H]{\Dif F(u_\varepsilon(s))}{B_\varepsilon(s) \dW(s)} ~ \P\text{-a.s.} ~\forall t \in [0,T].
        \end{split}
      \end{align}

      Because of the assumptions on $F$ and an infinite dimensional version of the dominated convergence theorem for stochastic integrals \cite[Theorem IV.32]{Protter2004} we can pass to the limit $\varepsilon \to 0$ on both sides of this equation. 

      Hence, $F(u)$ has a continuous version for which
      \begin{align*}
        F(u(t)) - F(u_0) &= \int_0^t \sprod[V',V]{v(s)}{\Dif F(u(s))} \ds \\
                         &\phantom{{}={}} +  \frac{1}{2} \int_0^t \tr \left[ \hess F(u(s)) B(s) (B(s))^\ast \right] \ds  \\
                         &\phantom{{}={}} + \int_0^t \sprod[H]{\Dif F(u(s))}{B(s) \dW(s)} ~\forall t \in [0,T].
      \end{align*}

      \emph{Step 4: $u$ has a continuous version.}\\
      We infer from the calculations above that there is a version of $u$ such that $\| u \|_H^2$ is continuous and
      \begin{align*}
        \| u(t) \|_H^2 - \| u_0 \|_H^2 &= \int_0^t 2 \sprod[V',V]{v(s)}{u(s)} + \| B(s) \|_{L_2(U;H)}^2 \ds \\
                                       &+ 2 \int_0^t \sprod[H]{u(s)}{B(s) \dW(s)} ~\forall t \in [0,T].
      \end{align*}
      From \eqref{eq:ito-du} and \eqref{eq:ito-sup-H} we infer that $u \in C([0,T]; (H, w))$ a.s. It is well known (cf. \cite[Theorem 4.2.5]{Prevot2007}) that this together with the continuity of $\| u \|_H^2$ implies that $u \in C([0,T]; H)$ a.s.  Since $H$ is separable we can apply \cite[Proposition 3.18]{DaPrato2014} to conclude that $u : \Omega \to C([0,T];H)$ is measurable. This proves that $u \in L^2(\Omega; C([0,T]; H))$.
    \end{proof}

  \begin{lemma}
    \label{lem:variational-continuity-functional}
    Under the assumptions of \Cref{prop:variational-ito} we have $F|_V \in C^2(V)$ and $F|_V$, $\Dif F|_V : V \to V'$ and $\hess F|_V : V \to L(V;V')$ are uniformly continuous on bounded subsets of $V$.

    \begin{proof}
      We only have to prove the continuity of $\hess F|_V : V \to L(V;V')$ and the uniform continuity on bounded subsets of $V$.

      The compactness of the embeddings $V\subset H \simeq H' \subset V'$ implies that the embedding $L(H) \subset L(V;V')$ is compact. Thus, when $u_k \rightharpoonup u$ in $V$ then $u_k \to u$ in $H$ and by the assumptions from \Cref{prop:variational-ito} we infer $\hess F(u_k) \overset{\ast}{\rightharpoonup} \hess F(u)$ in $L(H)$, hence $\hess F(u_k) \to \hess F(u)$ in $L(V;V')$. This proves that $\hess F|_V : (V,w) \to L(V;V')$ is continuous.
      
      Let $M \subset V$ be a bounded set in $V$, then $M$ is precompact in $H$ and therefore $F|_M : M \to \R$ and $\Dif F|_M : M \to H \subset V'$ are uniformly continuous. Furthermore $M$ is precompact in $(V,w)$ and therefore $\hess F|_M : M \to L(V;V')$ is uniformly continuous.
    \end{proof}
  \end{lemma}

  We will apply \Cref{prop:variational-ito} to the appropriate spaces for \eqref{eq:smcf}.
  \begin{corollary}
    \label{cor:variational-sobolev-ito}
    \begin{sloppypar}Let $T > 0$, $\stochbasis$ be a stochastic basis with a normal filtration and $W$ a cylindrical Wiener process on $U$.
    Furthermore let $u_0 \in L^2(\Omega; H^1(\T^n))$ be $\F_0$-measurable and $u, v, B$ be predictable processes with $u \in L^2(\Omega; L^2(0,T;H^2(\T^n)))$, $v \in L^2(\Omega; L^2(0,T;L^2(\T^n)))$ and $B \in L^2(\Omega; L^2(0,T;L_2(U; H^1(\T^n))))$ such that
    \begin{align}
      \label{eq:variational-sobolev-ito-du}
      \du = v \dt + B \dW \text{ in } L^2(\T^n).
    \end{align}
    \end{sloppypar}
    Then $u$ has a version with continuous paths in $H^1(\T^n)$ and for this version it holds that $u \in L^2(\Omega; C([0,T]; H^1(\T^n)))$.
    If $F = F(z,p) \in C^2(\R \times \R^n)$ with $\partial_z^2 F$, $\partial_z \nabla_p F$ and $\hess[p] F$ bounded 
    then we have
    \begin{align*}
      \int_{\T^n} &F(u(t), \nabla u(t)) \dx - \int_{\T^n} F(u_0, \nabla u_0) \dx \\
      &= \int_0^t \int_{\T^n} \left( \partial_z F(u(s), \nabla u(s)) - \divergence \left(\nabla_p F(u(s), \nabla u(s))\right) \right) v(s) \ds \\
                         &\phantom{{}={}} + \frac{1}{2} \sum_{l\in\N} \int_0^t \int_{\T^n} \partial_{zz} F(u(s), \nabla u(s)) \left|B_l(s) \right|^2 \ds \\
                         &\phantom{{}={}}+ \frac{1}{2} \sum_{l\in\N} \int_0^t \int_{\T^n} \partial_z \nabla_p F(u(s), \nabla u(s)) \cdot \nabla \left( \left| B_l(s) \right|^2 \right) \ds \\
                         &\phantom{{}={}}+ \frac{1}{2} \sum_{l\in\N} \int_0^t \int_{\T^n} \nabla B_l(s)  \cdot \hess[p] F(u(s), \nabla u(s)) \nabla  B_l(s) \ds \\
                         &\phantom{{}={}}+ \sum_{l\in\N} \int_0^t \int_{\T^n} \big[ \partial_z F(u(s), \nabla u(s)) B_l(s) \\
                         &\phantom{{}={}}\qquad\qquad\quad\quad + \nabla_p F(u(s), \nabla u(s)) \cdot \nabla B_l(s) \big] \dif\beta_l(s)
    \end{align*}
    a.s. for all $t \in [0,T]$, where $B_l\coloneqq Bg_l$, $l \in \N$.
    \begin{proof}
      We consider the spaces $V = H^2(\T^n)$ and $H = H^1(\T^n)$. To work in the framework from above we have to do the rather unusual identification of $w \in H^1(\T^n)$ with $J_Hw \coloneqq -\lapl w + w \in H'$ where
      \begin{align*}
        \sprod[H',H]{J_Hw}{\varphi} = \sprod[H^1(\T^n)]{w}{\varphi} = \int_{\T^n} \nabla w \cdot \nabla \varphi + w \varphi ,~\varphi \in H^1(\T^n).
      \end{align*}
      Then
      \begin{align}
        \label{eq:variational-sobolev-ito-du2}
        \begin{split}
          \dif \sprod[H',H]{J_Hu}{w} &= \dif \sprod[L^2(\T^n)]{u}{-\lapl w + w} \\
          &= \sprod[L^2(\T^n)]{v}{-\lapl w + w} \dt  \\
        &\phantom{{}={}}+ \sprod[L^2(\T^n)]{B \dW}{-\lapl w + w}  ~\forall w \in H^2(\T^n),
        \end{split}
      \end{align}
      which is an equation for $J_H u$ in $V'$.
      We consider the function $G : H^1(\T^n) \to \R$ with
      \begin{align*}
        G(w) \coloneqq \int_{\T^n} F(w(x), \nabla w(x)) \dx, ~w \in H^1(\T^n).
      \end{align*}
      Since $F \in C^2$ it easy to check that $G \in C^1(H^1(\T^n))$ and that the second G\^ateaux derivative $\hess G$ exists.
      We calculate for $w, \varphi, \psi \in H^1(\T^n)$
      \begin{align*}
        \sprod[H',H]{\Dif G(w)}{\varphi} &= \int_{\T^n} \partial_z F(w,\nabla w) \varphi + \nabla_p F(w, \nabla w) \cdot \nabla \varphi, \\
        \sprod[H',H]{\hess G(w) \varphi}{\psi} &= \int_{\T^n} \partial_{zz} F(w, \nabla w) \varphi \psi + \partial_z \nabla_p F(w, \nabla w) \cdot \left(\varphi \nabla \psi + \psi \nabla \varphi\right) \\
        &\phantom{{}={}} + \int_{\T^n} \nabla \varphi \cdot \hess[p] F(w, \nabla w) \nabla \psi.
      \end{align*}
      We have that $G$ and $\Dif G$ are bounded on bounded subsets of $H^1(\T^n)$ and that $\hess G$ is bounded because of the bounds of the second derivatives of $F$. On bounded subsets of $L(H) = \left(L_1(H)\right)^\ast$ the weak-$\ast$ topology is equivalent to the weak operator topology and therefore the continuity of $\hess G : H \to (L(H), w^\ast)$ follows from the fact that for all $w_k \to w$ in $H^1(\T^n)$ and all $\varphi, \psi \in H^1(\T^n)$ we have
      \begin{align*}
        &\sprod[H',H]{\hess G(w_k) \varphi}{\psi} \\
        &= \int_{\T^n} \partial_{zz} F(w_k, \nabla w_k) \varphi \psi + \partial_z \nabla_p F(w_k, \nabla w_k) \cdot \left(\varphi \nabla \psi + \psi \nabla \varphi\right) \\
                                                 &\to \int_{\T^n} \partial_{zz} F(w, \nabla w) \varphi \psi + \partial_z \nabla_p F(w, \nabla w) \cdot \left(\varphi \nabla \psi + \psi \nabla \varphi\right) \\
                                                 &= \sprod[H',H]{\hess G(w) \varphi}{\psi}.
      \end{align*}

      For $w \in H^2(\T^n), \varphi \in H^1(\T^n)$ we have
      \begin{align*}
        &\sprod[H',H]{\Dif G(w)}{\varphi} \\
        &= \int_{\T^n} \left( \partial_z F(w, \nabla w) - \partial_z \nabla_p F(w, \nabla w) \cdot \nabla w - \hess[p] F(w, \nabla w) : \hess w \right) \varphi \\
        &= \int_{\T^n} \Phi(w) \varphi,
      \end{align*}
      with
      \begin{align*}
        \Phi(w) \coloneqq \partial_z F(w, \nabla w) - \partial_z \nabla_p F(w, \nabla w) \cdot \nabla w - \hess[p] F(w, \nabla w) : \hess w.
      \end{align*}
      Because of the assumptions on $F$ we find that $\Phi(w) \in L^2(\T^n)$ for $w \in H^2(\T^n)$ and $\Phi : H^2(\T^n) \to L^2(\T^n)$ is continuous.
      Since for the restriction of $J_H$ to $H^2(\T^n)$ we have that $J_H|_{H^2(\T^n)} : H^2(\T^n) \to L^2(\R^n)$ is an isomorphism, we conclude $J_H^{-1} \circ \Phi(w) \in H^2(\T^n)$ and $J_H^{-1} \circ \Phi : H^2(\T^n) \to H^2(\T^n)$ is continuous with
      \begin{align*}
        \| J_H^{-1} \Phi(w) \|_{H^2(\T^n)} \le C \| \Phi(w) \|_{L^2(\T^n)} \le C(1 + \| w \|_{H^2(\T^n)}).
      \end{align*}
      Note that for the application of \Cref{prop:variational-ito} we shall have an equation for $\du$ in $V'$, whereas \eqref{eq:variational-sobolev-ito-du} is an equation for $\du$ in $L^2(\T^n)$. Therefore we have to use \eqref{eq:variational-sobolev-ito-du2} to infer that a.s. for all $t \in [0,T]$
      \begingroup
      \allowdisplaybreaks
      \begin{align*}
        G(u&(t)) - G(u_0) \\&= \int_0^t \sprod[V',V]{J_H v(s)}{J_H^{-1} \circ \Phi(u(s))} + \frac{1}{2} \tr \left[\hess G(u(s)) B(s) (B(s))^\ast \right] \ds \\
                         &\phantom{{}={}} + \int_0^t \sprod[H',H]{\Dif G(u(s))}{B(s) \dW(s)} \\
                         &= \int_0^t \int_{\T^n} \left( \partial_z F(u(s), \nabla u(s)) - \divergence \left(\nabla_p F(u(s), \nabla u(s))\right) \right) v(s) \ds \\
                         &\phantom{{}={}} + \frac{1}{2} \sum_{l\in\N} \int_0^t \int_{\T^n} \partial_{zz} F(u(s), \nabla u(s)) \left|B(s) g_l\right|^2 \ds \\
                         &\phantom{{}={}}+ \frac{1}{2} \sum_{l\in\N} \int_0^t \int_{\T^n} \partial_z \nabla_p F(u(s), \nabla u(s)) \cdot \nabla \left( \left| B(s)g_l \right|^2 \right) \ds \\
                         &\phantom{{}={}}+ \frac{1}{2} \sum_{l\in\N} \int_0^t \int_{\T^n} \nabla \left( B(s) g_l \right) \cdot \hess[p] F(u(s), \nabla u(s)) \nabla \left( B(s) g_l \right) \ds \\
                         &\phantom{{}={}}+ \sum_{l\in\N} \int_0^t \int_{\T^n} \big[\partial_z F(u(s), \nabla u(s)) B(s) g_l \\
                         &\hspace{3cm}+ \nabla_p F(u(s), \nabla u(s)) \cdot \nabla \left(B(s) g_l \right) \big]\dif\beta_l(s).
      \end{align*}
      \endgroup
    \end{proof}
  \end{corollary}

  \subsection{Existence for variational SPDEs}
  \label{sec:variational-existence}
  We will adapt the approach of \cite[Section 2.3.3]{Pardoux2007} to prove existence of weak solutions for variational SPDEs
  \begin{align}
    \label{eq:variational-spde}
    \begin{split}
      \du &= A(u) \dt + B(u) \dW \\
      u(0) &= u_0,
    \end{split}
  \end{align}
  which goes back to \cite{Viot1976}.

  In addition to \Cref{assumptions:variational} we will make the following assumptions.
  \begin{assumptions}
    \label{assumptions:variational-existence}
    Let $A : V \to V'$ and $B : V \to L_2(U;H)$. We will write $B^\ast : V \to L_2(H;U)$ for the adjoint operator $B^\ast(u) \coloneqq \left(B(u)\right)^\ast$. 
    We assume:
    \begin{itemize}
    \item \textbf{Coercivity:} There are constants $\alpha, C > 0$ such that
      \begin{align}
        \label{eq:variational-coercivity}
        2\sprod[V', V]{A(u)}{u} + \| B(u) \|_{L_2(U;H)}^2 \le -\alpha \| u \|_{V}^2 + C\left(1 + \| u \|_{H}^2\right) ~\forall u \in V.
      \end{align}
    \item \textbf{Growth bounds:} There is a constant $C > 0$ and $\delta \in (0,2]$ such that
      \begin{align}
        \label{eq:variational-growth-A}
        \| A(u) \|_{V'}^2 &\le C\left(1+ \|u\|_V^2\right) ~\forall u \in V,\\
        \label{eq:variational-growth-B}
        \| B(u) \|_{L_2(U;H)}^2 &\le C\left(1 + \| u \|_V^2 \right) ~\forall u \in V,\\
        \label{eq:variational-growth-B-to-V'}
        \| B(u) \|_{L(U;V')}^2 &\le C\left(1 + \|u \|_V^{2-\delta} + \|u\|_H^2\right).
      \end{align}
    \item \textbf{Continuity:} $A : V \to V'$ is weak-weak-$\ast$ sequentially continuous, that means
      \begin{align}
        \label{eq:variational-continuity-A}
        u_k \rightharpoonup u \text{ in } V ~\Rightarrow~ A(u_k) &\overset{\ast}{\rightharpoonup} A(u) \text{ in } V'
      \end{align}
      and $B^\ast : V \to L_2(H ; U)$ is sequentially continuous from the weak topology on $V$ to the strong operator topology on $L(H;U)$, that means
      \begin{align}
        \label{eq:variational-continuity-B}
        u_k \rightharpoonup u \text{ in } V ~\Rightarrow~ B^\ast(u_k) h &\to B^\ast(u) h \text{ in } U ~\forall h \in H.
      \end{align}
    \end{itemize}
  \end{assumptions}

  The assumptions \eqref{eq:variational-coercivity}, \eqref{eq:variational-growth-A} and \eqref{eq:variational-continuity-A} are the same as in \cite{Pardoux2007}, whereas \eqref{eq:variational-continuity-B} is weaker.
  Furthermore we have replaced the sublinear growth bound from \cite{Pardoux2007} for $B(u)$ by the weaker assumptions \eqref{eq:variational-growth-B} and \eqref{eq:variational-growth-B-to-V'}.
  These weaker assumptions are necessary to apply the theory to the viscous equation \eqref{eq:smcf-viscous}. To prove this generalization we have to prove bounds for higher moments of the $\| \cdot \|_H$ norm of the approximations, whereas in the proof in \cite{Pardoux2007} only the second moment of the $\| \cdot \|_H$ norm needed to bounded. This will be done in \Cref{prop:variational-bounds} under the additional assumption that the corresponding higher moment of the $\| \cdot \|_H$ norm is  bounded for the initial data.
  Similarly to the ideas of \cite{Hofmanova2017}, we will use the Jakubowski-Skorokhod representation theorem \cite{Jakubowski1997} for tight sequences in non-metric spaces to prove that our approximations converge on a different probability space. We will make use of similar arguments as in \cite{Breit2018} to handle the unbounded time interval. Finally, we will show that this limit is a martingale solution of \eqref{eq:variational-spde} using a general method of constructing martingale solutions without relying on any kind of martingale representation theorem, which was introduced in \cite{Brzezniak2007} and already used in \cite{Ondrejat2010} and \cite{Hofmanova2017}, among others.
  
  We will use a standard Galerkin scheme (compare with \cite[Chapter 2.3]{Pardoux2007}) to prove that there is a martingale solution of \eqref{eq:variational-spde} if the initial condition has bounded $q$-th moment in $H$ for some $q > 2$. With the $(e^k)_{k\in\N}$ as in \Cref{assumptions:variational} we will write
  \begin{align*}
    V_N \coloneqq \vspan \left(\{ e^1, \dotsc, e^N \}\right), ~N \in \N.
  \end{align*}
  Our main result is:

  \begin{theorem}
    \label{thm:variational-existence}
    Let $q > 2$ and $\Lambda$ be a Borel probability measure on $H$ with finite $q$-th moment
    \begin{align*}
      \int_H \| z \|_H^q \dif\Lambda(z) < \infty.
    \end{align*}
    Then there is a martingale solution of \eqref{eq:variational-spde} with initial data $\Lambda$. That means, that there is a stochastic basis $\stochbasis[\infty]$ with a normal filtration, a cylindrical $(\F_t)$-Wiener process $W$ on $U$ and a predictable $u$ with $u \in L^2(\Omega;L^2(0,T;V)) \cap L^2(\Omega; C([0,T]; H))$ for all $T > 0$ and
    \begin{align*}
    &  \sprod[H]{u(t)}{v} - \sprod[H]{u(0)}{v} = \int_0^t \sprod[V',V]{A(u(s))}{v} \ds + \int_0^t \sprod[H]{B(u(s)) \dW(s)}{v}  \\
                                              &\qquad= \int_0^t \sprod[V',V]{A(u(s))}{v} \ds + \sum_{l \in \N} \int_0^t \sprod[H]{B(u(s))g_l}{v} \dif\beta_l(s) 
    \end{align*}
    $\P$-a.s. for all $t \in [0,\infty)$ and $v \in V$, and
    \begin{align*}
      \P \circ u(0)^{-1} = \Lambda.
    \end{align*}
  \end{theorem}

  To prove \Cref{thm:variational-existence}, we will consider \eqref{eq:variational-spde} on the finite-dimensional space $V_N$.
  \begin{theorem}
    \label{thm:variational-sde-existence}
    Let $N \in \N$ and $\Lambda$ be a Borel probability measure on $H$. Then there is a weak solution of the finite-dimensional approximation of \eqref{eq:variational-spde}.

    That means, that there is a stochastic basis $\stochbasis[\infty]$ with a normal filtration, $\beta_1, \dotsc, \beta_N$ mutually independent real-valued $(\F_t)$-Brownian motions and a predictable $V_N$-valued process $u$ with $u \in L^2(\Omega; C([0,T]; V_N))$ for all $T > 0$ such that
    \begin{align*}
&      \sprod[H]{u(t)}{v} - \sprod[H]{u(0)}{v} \\
      &\qquad= \int_0^t \sprod[V',V]{A(u(s))}{v} \ds + \sum_{l=1}^N \int_0^t \sprod[H]{B(u(s))g_l}{v} \dif\beta_l(s)
    \end{align*}
    $\P$-a.s. for all $t \in [0,\infty)$ and $v \in V_N$, and
    \begin{align*}
      \P \circ u(0)^{-1} = \Lambda_N \coloneqq \Lambda \circ P_N^{-1},
    \end{align*}
    where $P_N : H \to V_N$ is the orthogonal projection with respect to $H$.

    \begin{proof}
      We transform the equation into an $N$-dimensional stochastic differential equation. For this let
      \begin{gather*}
        \tilde{A} : \R^N \to \R^N \\
        \left(\tilde{A}(\alpha) \right)_k \coloneqq \sprod[V',V]{A\left(\sum_{m=1}^N \alpha_m e^m \right)}{e^k}, ~k = 1, \dotsc, N, ~ \alpha \in \R^N
      \end{gather*}
      and
      \begin{gather*}
        \tilde{B} : \R^N \to \R^{N\times N} \\
        \left(\tilde{B}(\alpha)\right)^k_{l} \coloneqq \sprod[H]{B\left(\sum_{m=1}^N \alpha_m e^m \right) g_l}{e^k}, ~k,l=1,\dotsc, N, ~ \alpha \in \R^N.
      \end{gather*}
      These mappings are continuous and grow at most linearly. Therefore we can apply a classical theorem for finite dimensional stochastic differential equations \cite[Theorem 0.1]{Hofmanova2012} and \cite[Theorem IV.2.4]{Ikeda1981} to find a weak solution of
      \begin{align*}
        \dif \alpha_k &= \left(\tilde{A}(\alpha)\right)_k \dt + \sum_{l=1}^N \left(\tilde{B}(\alpha)\right)^k_l \dif \beta_l, ~ k=1,\dotsc,N\\
        \P \circ \alpha(0)^{-1} &= \tilde{\P}_N.
      \end{align*}
      where
      \begin{align*}
        \tilde{\P}_N(M) \coloneqq \Lambda_N\left( \left\{ \sum_{k=1}^N \alpha_k e^k \in H \, \Big| \,\alpha \in M \right\} \right), ~ M \subset \R^N.
      \end{align*}
      Defining $u(t) = \sum_{k=1}^N \alpha_k(t) e^k$ for $t \in [0,\infty)$ we find $u \in L^2(\Omega; C([0,T];V_N))$ for all $T > 0$ with
      \begin{align*}
        \sprod[H]{u(t)}{e^k} &- \sprod[H]{u_0}{e^k}
        = \alpha_k(t) - \alpha_k(0) \\
        &= \int_0^t \left(\tilde{A}(\alpha(s))\right)_k \ds + \sum_{l=1}^N \left(\tilde{B}(\alpha(s))\right)^k_l \dif \beta_l(s) \\
        &= \int_0^t \sprod[V',V]{A(u(s))}{e^k} \ds + \sum_{l=1}^N \int_0^t \sprod[H]{B(u(s))g_l}{e^k} \dif \beta_l(s)
      \end{align*}
      and
      \begin{align*}
        \P \circ u(0)^{-1} = \Lambda_N.
      \end{align*}
    \end{proof}
  \end{theorem}

  \begin{proposition}[Estimates for the norm]
    \label{prop:variational-bounds}
    Assume that $T > 0$ and $\stochbasis$ is a stochastic basis with a normal filtration.  Then there is a constant $C > 0$ that only depends on the constants from \Cref{assumptions:variational-existence}, such that for all mutually independent real-valued $(\F_t)$-Brownian motions $(\beta_l)_{l\in\N}$, $N \in \N$ and all $V_N$-valued predictable processes $u \in L^2(\Omega; C([0,T]; V_N))$ with
    \begin{align*}
&      \sprod[H]{u(t)}{v} - \sprod[H]{u(0)}{v} \\
      &\qquad= \int_0^t \sprod[V',V]{A(u(s))}{v} \ds + \sum_{l=1}^N \int_0^t \sprod[H]{B(u(s))g_l}{v} \dif \beta_l(s)
    \end{align*}
    $\P$-a.s. for all $t \in [0,T]$ and $v \in V_N$, we have
    \begin{align*}
      \E \sup_{t \in [0,T]} \| u(t) \|_H^2 + \E \int_0^T \| u(t) \|_V^2 \dt \le Ce^{CT}\left(1 + \E \|u(0)\|_H^2\right).
    \end{align*}
    Additionally, there is a $q_0 > 2$ such that $u(0) \in L^q(\Omega;H)$ for some $q \in (2,q_0)$ implies $u \in L^\infty(0,T;L^q(\Omega; H))$ with
    \begin{align*}
      \E \| u(t) \|_H^q \le e^{Ct}\left(1+ \E \| u(0) \|_H^q \right) ~\forall t \in [0,T].
    \end{align*}
    \begin{proof}
      From \Cref{prop:variational-ito} we conclude that the following It\^o formula holds for the norm of solutions
      \begin{align}
        \label{eq:variational-ito-norm}
        \begin{split}
          \dif \| u\|_{H}^2 &= 2 \sprod[V',V]{A(u)}{u} \dt + \sum_{k,l=1}^N  \sprod[H]{ B(u)g_l }{e^k}^2 \dt + \sum_{l=1}^N 2\sprod[H]{B(u)g_l}{u} \dif \beta_l \\
          &=2 \sprod[V',V]{A(u)}{u} \dt + \| B_N(u) \|_{L_2(U;H)}^2 \dt + 2 \sprod[H]{B_N(u) \dW}{u},
        \end{split}
      \end{align}
      where $B_N : V \to L_2(U;H)$ is $B$ restricted to the finite-dimensional subspaces,
      \begin{align*}
        \sprod[H]{B_N(u) g_l}{e^k} \coloneqq \left\{
        \begin{array}{cl}
          \sprod[H]{B(u) g_l}{e^k} & \text{if }  k \le N \text{ and } l \le N,\\
          0 & \text{otherwise,}
        \end{array}\right.
              k,l \in \N, ~ u \in V.
      \end{align*}
      For $q \ge 1$ we use the It\^o formula for real-valued semimartingales to deduce that
      \begin{align*}
        \dif \left(1+\| u \|_{H}^2\right)^{q} &= q \left(1+\| u \|_{H}^2\right)^{q-1} \left( 2 \sprod[V',V]{A(u)}{u} +  \| B_N(u) \|_{L_2(U;H)}^2 \right) \dt \\
                                              &\phantom{{}={}} + 2q(q-1) \left(1+\|u \|_{H}^2\right)^{q-2} \| B_N^\ast(u) u \|_{U}^2 \dt \\
                                              &\phantom{{}={}} + 2 q \left(1+\| u \|_H^2\right)^{q-1} \sprod[H]{B_N(u) \dW}{u}.
      \end{align*}
      By taking the expectation and using the coercivity \eqref{eq:variational-coercivity} as well as the growth bounds \eqref{eq:variational-growth-B} we conclude for $q \in [1,1+\varepsilon)$ with $\varepsilon < \frac{\alpha}{C}$ where $C$ depends on the constants from \eqref{eq:variational-coercivity} and \eqref{eq:variational-growth-B}, that
      \begin{align}
        \label{eq:variational-ito-norm-qth}
        \begin{split}
&          \E \left(1+\| u(t) \|_H^2\right)^{q} - \E \left(1+\| u(0) \|_H^2\right)^{q} \\
          &\le q \E \int_0^t \left(1+\| u(s) \|_H^2\right)^{q-1} \left( 2 \sprod[V',V]{A(u(s))}{u(s)} +  \| B_N(u(s)) \|_{L_2(U;H)}^2\right) \ds \\
          &\phantom{{}={}} + 2q(q-1) \E \int_0^t \left(1+\| u(s) \|_H^2\right)^{q-2} \| B_N^\ast(u(s))u(s) \|_U^2  \ds \\
          &\le q \E \int_0^t \left(1+ \| u(s) \|_H^2\right)^{q-1} \left( -\alpha \| u(s) \|_{V}^2 + C\left(1 + \| u(s) \|_{H}^2\right)\right) \ds \\
          &\phantom{{}={}} + Cq(q-1) \E \int_0^t \left(1+\| u(s) \|_H^2\right)^{q-1} \| B_N^\ast(u(s)) \|_{L_2(H;U)}^2 \ds \\
          &\le -q\left(\alpha - C(q-1)\right)\E \int_0^t \left(1+\| u(s) \|_H^2\right)^{q-1}\| u(s) \|_V^2 \ds \\
          &\phantom{{}={}}+ C \E \int_0^t \left(1 + \| u(s)\|_H^2\right)^{q} \ds \\
          &\le C \E \int_0^t \left(1+ \| u(s) \|_H^2\right)^q \ds ~\forall t \in [0,T]
        \end{split}
      \end{align}
      and with a Gronwall argument
      \begin{align*}
        \E \left(1+\|u(t)\|_H^2\right)^{q} \le e^{Ct} \E \left(1 + \| u(0) \|_H^2\right)^{q} ~\forall t \in [0,T].
      \end{align*}
      This already implies that there is a constant $C > 0$ such that
      \begin{align*}
        \E \| u(t) \|_H^{2q} \le e^{Ct} \left(1 + \E \| u(0) \|_H^{2q} \right) ~\forall t \in [0,T].
      \end{align*}
      Furthermore, we have for the stochastic integral in \eqref{eq:variational-ito-norm} using the Burkholder-Davis-Gundy inequality \cite[Theorem 3.15]{DaPrato2014} and \eqref{eq:variational-growth-B}
      \begin{align}
        \label{eq:variational-sup-stoch}
        \begin{split}
        \E \sup_{t \in [0,T]} \Bigg| 2 \int_0^t &\sprod[H]{B_N(u) \dW}{u} \Bigg|
        \le C \E \left[ \int_0^T \left\| B^\ast_N(u(s))u(s) \right\|_U^2 \ds \right]^{\frac{1}{2}} \\
        &\le C \E \left[ \sup_{t \in [0,T]} \| u(s) \|_H^2 \left(1 + \int_0^T \| u(s) \|_V^2 \ds\right)\right]^{\frac{1}{2}} \\
        &\le  \frac{1}{2} \E \sup_{t \in [0,T]} \| u(s) \|_H^2 + C \left(1+\E \int_0^T \| u(s)\|_V^2 \ds\right).
        \end{split}
      \end{align}
      And from \eqref{eq:variational-ito-norm-qth} for $q=1$ we infer
      \begin{align*}
        \E \int_0^T \| u(s) \|_V^{2} \ds  \le e^{CT} \left(1+\E \| u(0) \|_H^2\right),
      \end{align*}
      hence with \eqref{eq:variational-sup-stoch}
      \begin{align*}
      &  \E \sup_{t \in [0,T]} \| u(t) \|_H^2 \le \E \| u(0) \|_H^2 \\
                                             &\phantom{{}={}}+ \E \sup_{t \in [0,T]} \int_0^t \left[2 \sprod[V',V]{A(u(\tau))}{u(\tau)} + \| B_N(u(\tau)) \|_{L_2(U;H)}^2\right] \dif \tau \\
                                             &\phantom{{}={}} + \E \sup_{t \in [0,T]}  2 \int_0^t \sprod[H]{B_N(u(\tau)) \dW(\tau)}{u(\tau)}  \\
                                             &\le \E \|u(0)\|_H^2  + \frac{1}{2} \E\sup_{t \in [0,T]} \| u(t) \|_H^2 + C\left(1+\E \int_0^T \| u(t) \|_V^2 \dt\right) \\
                                             &\le Ce^{CT}\left(1 + \E \| u(0) \|_H^2 \right) + \frac{1}{2} \E\sup_{t \in [0,T]} \| u(t) \|_H^2
      \end{align*}
      and therefore
      \begin{align*}
        \E \sup_{t \in [0,T]} \| u(t) \|_H^2 \le Ce^{CT}\left(1 + \E \| u(0) \|_H^2 \right).
      \end{align*}
    \end{proof}
  \end{proposition}

  \begin{lemma}
    \label{lem:variational-continuity}
    Let $T > 0$ and
    \begin{align*}
      \mathcal{X}_u \coloneqq (L^2(0,T;V),w) \cap L^2(0,T;H) \cap C([0,T];(H,w))
    \end{align*}
    with $C([0,T]; (H,w))$ endowed with the compact-open topology.
    Then for each $v \in V$ the mappings $\mathcal{A} : \mathcal{X}_u \to L^p(0,T)$ for $p < 2$ and $\mathcal{B}^\ast : \mathcal{X}_u \to L^2(0,T;U)$ with
    \begin{align*}
      (\mathcal{A}(u))(t) &\coloneqq \sprod[V',V]{A(u(t))}{v}  \text{ and}\\
      (\mathcal{B}^\ast(u))(t) &\coloneqq B^\ast(u(t))v, ~t \in [0,T]
    \end{align*}
    are sequentially continuous.

    \begin{proof}
      
      Let $(u_k)_k \subset \mathcal{X}_u$ be a sequence with $u_k \to u$ in $\mathcal{X}_u$.
      For $M > 0$ we consider the functions
      \begin{align*}
        u_k^M(t) \coloneqq \left\{
        \begin{array}{cl}
          u_k(t) &\text{if } \| u_k(t) \|_V \le M, \\
          u(t) &\text{otherwise}
        \end{array}
                 \right., ~t\in[0,T], k \in \N.
      \end{align*}
      Since $u \in L^2(0,T;V)$ we conclude that for almost every $t \in [0,T]$ the sequence $(u_k^M(t))_{k\in\N}$ is uniformly bounded in $V$. Furthermore we know for every $t \in [0,T]$ that $u_k^M(t) \rightharpoonup u(t)$ in $H$, because $u_k \to u$ in $\mathcal{X}_u$ implies
      \begin{align*}
        \left| \sprod[H]{u_k^M(t) - u(t)}{h} \right| & \le \left| \sprod[H]{u_k(t) - u(t)}{h} \right| \to 0 \text{ as } k\to\infty ~\forall h \in H.
      \end{align*}
      Hence $u_k^M(t) \rightharpoonup u(t)$ in $V$ for $k \to \infty$ for almost every $t \in [0,T]$.

      The continuity assumptions \eqref{eq:variational-continuity-A} on $A$ and \eqref{eq:variational-continuity-B} on $B$ imply
      \begin{align*}
        (\mathcal{A}(u_k^M))(t) \to (\mathcal{A}(u))(t), \\
        (\mathcal{B}^\ast(u_k^M))(t) \to (\mathcal{B}^\ast(u))(t)
      \end{align*}
      for almost every $t \in (0,T)$. Furthermore, using \eqref{eq:variational-growth-A} we get
      \begin{align*}
        \left|(\mathcal{A}(u_k^M))(t)\right|^2 &\le C\left(1 + \|u_k(t)\|_V^2\right) \| v \|_V^2.
      \end{align*}
      Now, with Vitali's convergence theorem we infer that $\mathcal{A}(u_k^M) \to \mathcal{A}(u)$ in $L^p(0,T)$ for all $p < 2$.

      For $\mathcal{B}^\ast$ we have with the growth bound \eqref{eq:variational-growth-B-to-V'}
      \begin{align*}
&        \Big\| (\mathcal{B}^\ast(u_k^M))(t) - (\mathcal{B}^\ast(u))(t) \Big\|_U^2 
       \\
       & \le \left( \| B(u_k^M(t)) \|^2_{L(U;V')} + \| B(u(t)) \|^2_{L(U;V')} \right) \| v \|_V^2 \\
        &\le C\left(1 + \| u_k(t) \|_V^{2-\delta} + \| u(t) \|_V^{2-\delta} + \| u_k(t) \|_H^2 + \| u(t) \|_H^2 \right) \| v \|_V^2.
      \end{align*}
      The right hand side is uniformly integrable, because $\|u_k(t)\|_V^{2-\delta}$ is bounded in $L^{\frac{2}{2-\delta}}(0,T)$ and $\| u_k(t) \|_H^2$ is convergent in $L^1(0,T)$. Therefore by Vitali's convergence theorem $\mathcal{B}^\ast(u_k^M) \to \mathcal{B}^\ast(u)$ in $L^2(0,T;U)$.

      Let $E_k^M \coloneqq \{ t \in [0,T] \mid \| u_k(t) \|_V > M \}$ for $k \in \N$. For the measure of $E_k^m$ we estimate
      \begin{align*}
        \left| E_k^m \right| \le \int_0^T \frac{\| u_k(t) \|_V^2}{M^2} \dt \le \frac{C}{M^2}, 
      \end{align*}
      because $(u_k)_{k\in\N}$ is uniformly bounded in $L^2(0,T;V)$.

      As above one can conclude from the growth assumptions \eqref{eq:variational-growth-A} and \eqref{eq:variational-growth-B-to-V'} and the fact that $\| u_k^M(t) \|_V \le \| u_k(t) \|_V + \| u(t) \|_V$ and $\| u_k^M(t) \|_H \le \| u_k(t) \|_H + \| u(t) \|_H$ that
      \begin{gather*}
        \left| \mathcal{A}(u_k) - \mathcal{A}(u_k^M) \right|^p \text{ and }
        \left\| \mathcal{B}^\ast(u_k) - \mathcal{B}^\ast(u_k^M) \right \|_U^2
      \end{gather*}
      are uniformly integrable with respect to $k$ and $M$. Hence,
      \begin{align*}
        &\left\| \mathcal{A}(u_k) - \mathcal{A}(u) \right\|_{L^p(0,T)} \\
        &\qquad\le \left\| \mathcal{A}(u_k) - \mathcal{A}(u_k^M) \right\|_{L^p(E_k^M)} + \left\| \mathcal{A}(u_k^M) - \mathcal{A}(u) \right\|_{L^p(0,T)}
              \end{align*}
and
         \begin{align*}
      &  \left\| \mathcal{B}^\ast(u_k) - \mathcal{B}^\ast(u) \right\|_{L^2(0,T;U)} \\
      &\qquad\le  \left\| \mathcal{B}^\ast(u_k) - \mathcal{B}^\ast(u_k^M) \right\|_{L^2(E_k^M;U)} + \left\| \mathcal{B}^\ast(u_k^M) - \mathcal{B}^\ast(u) \right\|_{L^2(0,T;U)}
      \end{align*}
      converge to $0$ by first choosing $M$ large such that the first terms on the right hand side become small and then choosing $k$ large and using the convergences derived above.
    \end{proof}
  \end{lemma}

  \begin{proof}[Proof of \Cref{thm:variational-existence}]
    For $N \in \N$ let $V_{N}:=  \vspan(\{e^1, \dotsc, e^N \})$ and consider the $V_N$-valued process $u^N$ from \Cref{thm:variational-sde-existence}. The process $u^N$ is a weak solution of the finite-dimensional approximation of \eqref{eq:variational-spde} for a Wiener process $W^N$ on $U$ with covariance operator $Q_N : U \to \vspan \left(\{ g_1, \dotsc, g_N \}\right)$, which is the orthogonal projection. We can assume that the processes $(u^N)_{N \in \N}$ are defined on one common probability space $(\Omega, \F, \P)$ = $([0,1], \mathcal{B}([0,1]), \mathcal{L})$, because the proof of \cite[Theorem 0.1]{Hofmanova2012} could be adapted to yield existence of weak solutions for the finite-dimensional approximation on this particular space. (cf. \cite[Theorem IV.2.3 and Theorem IV.2.4]{Ikeda1981})

    Furthermore, we can always assume that $q > 2$ is sufficiently small such that the following arguments hold.  We can apply \Cref{prop:variational-bounds} to infer that for all $T > 0$ the sequence $(u^N)_{N \in \N}$ is in $N \in \N$ uniformly bounded in
    \begin{align*}
      L^2(\Omega ; C([0,T]; H)) \cap L^2(\Omega ; L^2(0,T;V)) \cap L^\infty(0,T;L^q(\Omega; H)).
    \end{align*}
    Let $Z$ be another separable Hilbert space with a Hilbert-Schmidt embedding $V' \subset Z$. 
    \begin{sloppypar}Because of \eqref{eq:variational-growth-B-to-V'} we have that $B_N(u^N)$ is uniformly bounded in $L^q(\Omega; L^q(0,T;L_2(U;Z)))$, since
    \begin{align*}
      \E \int_0^T \| B_N(u^N(t)) \|_{L_2(U;Z)}^q \dt
      &\le C \E \int_0^T \| B_N(u^N(t)) \|_{L(U;V')}^q \dt \\
      &\hspace{-5mm}\le C\left(1 + \E \int_0^T \| u^N(t) \|_V^2 + \sup_{t \in [0,T]} \E \| u^N(t) \|_H^q\right).
    \end{align*}
    \end{sloppypar}
    Using the factorization method \cite[Theorem 1.1]{Seidler1993} we get a uniform bound for $u^N \in L^2(\Omega; C^{0,\lambda}([0,T];Z))$ for some $\lambda > 0$.

    Now, consider another separable Hilbert space $U_1$ which is the completion of $U$ with the respect to the scalar product $\sprod[U_1]{g_{l_1}}{g_{l_2}} = a_{l_1}^2\sprod[U]{g_{l_1}}{g_{l_2}}$ for $l_1, l_2 \in \N$ and $(a_l)_{l \in \N} \subset \R$ a square-summable sequence. Then $U$ is densely embedded in $U_1$ with a Hilbert-Schmidt embedding and each $W^N$ can be understood as a Wiener process on $U_1$ with covariance operators uniformly bounded in $L_1(U_1)$. Hence, with a factorization argument for $\lambda \in (0,\frac{1}{2})$ the $(W^N)_N$ are uniformly bounded in $L^2(\Omega; C^{0,\lambda}([0,T];U_1))$. For $\lambda > 0$ the embeddings
    \begin{align*}
      C^{0,\lambda}([0,T]; Z) \cap C([0,T];H) &\to C([0,T]; (H,w)) \text{ and} \\
      C^{0,\lambda}([0,T]; Z) \cap L^2(0,T;V) &\to L^2(0,T;H)
    \end{align*}
    are compact because of \cite[Theorem 5]{Simon1987} and the Ascoli theorem \cite[Theorem 7.17]{Kelley1955}. And also the embedding
    \begin{align*}
      C^{0,\lambda}([0,T]; U_1) &\to C([0,T]; (U_1, w))
    \end{align*}
    is compact because of the Ascoli theorem \cite[Theorem 7.17]{Kelley1955}. Thus, the joint laws of $(u^N, W^N)$ are tight in $\mathcal{X}_u^T \times \mathcal{X}_W^T$     with
    \begin{align*}
      \mathcal{X}_u^T &\coloneqq  C([0,T]; (H,w))  \cap L^2(0,T;H) \cap \left(L^2(0,T;V), w\right) \text{ and}\\
      \mathcal{X}_W^T &\coloneqq C\left([0,T] ; (U_1, w)\right).
    \end{align*}
    Since this holds for all $T > 0$ and since a set is compact in $\mathcal{X}_u \times \mathcal{X}_W$ with
    \begin{align*}
      \mathcal{X}_u &\coloneqq C_{\text{loc}}([0,\infty); (H,w))  \cap L^2_{\text{loc}}(0,\infty;H) \cap \left(L^2_{\text{loc}}(0,\infty;V), w\right) \text{ and }\\
      \mathcal{X}_W &\coloneqq C_{\text{loc}}\left([0,\infty) ; (U_1, w)\right),
    \end{align*}
    where $C_{\text{loc}}([0,\infty); (H,w))$ and $C_{\text{loc}}([0,\infty); (U_1, w))$ are endowed with the compact-open topology,
    if and only if for all $T > 0 $ the set (with all of its elements restricted to $[0,T]$) is compact in $\mathcal{X}_u^T \times \mathcal{X}_W^T$, we conclude similarly to \cite[Proof of Proposition 4.3]{Breit2018} that the joint laws of $(u^N, W^N)$ are tight in $\mathcal{X}_u \times \mathcal{X}_W$.
      
    Because of \Cref{lem:jakubowski-topological-assumption} we can apply the Jakubowski-Skorokhod representation theorem for tight sequences in nonmetric spaces \cite[Theorem 2]{Jakubowski1997} to deduce the existence of a probability space $(\tilde{\Omega}, \tilde{\F}, \tilde{\P})$, an strictly increasing sequence $(N_m)_{m \in \N} \subset \N$, $\mathcal{X}_u$-valued random variables $\tilde{u}^m$, $\tilde{u}$ and $\mathcal{X}_W$-valued random variables $\tilde{W}^m$, $\tilde{W}$ for $m \in \N$ such that
    \begin{align*}
      \tilde{u}^m \to \tilde{u} ~&\tilde{\P}\text{-a.s. in } \mathcal{X}_u, \\
      \tilde{W}^{m} \to \tilde{W} ~&\tilde{\P}\text{-a.s. in } \mathcal{X}_W
    \end{align*}
    and the joint law of $(\tilde{u}^m, \tilde{W}^{m})$ coincides with the joint law of $(u^{N_m}, W^{N_m})$ for all $m \in \N$.  To simplify the notation, we will assume that $N_m = m$ for $m \in \N$.

    Let $(\G_t)_{t \in [0,\infty)}$ be the natural filtration of the process $(\tilde{u}, \tilde{W})$. That means $\G_t$ for $t \in [0,\infty)$ is the smallest $\sigma$-algebra such that $\tilde{u}(s) : \tilde{\Omega} \to H$ and $\tilde{W}(s) : \tilde{\Omega} \to U_1$ are measurable for all $s \in [0,t]$. The Pettis measurability theorem implies that for the Borel $\sigma$-algebras on $H$ and $U_1$ we have $\B((H, \| \cdot \|_H)) = \B((H, w))$ and $\B((U_1, \| \cdot \|_{U_1})) = \B((U_1, w))$.
    Therefore $\G_t$ is contained in the $\sigma$-algebra generated by $\tilde{u}|_{[0,t]} : \tilde{\Omega} \to C([0,t]; (H, w))$ and $\tilde{W}|_{[0,t]} : \tilde{\Omega} \to C([0,t]; (U_1,w))$. Choosing dense subsets of $[0,t]$ and $H$ respectively $U_1$, one can also show that $\tilde{u}|_{[0,t]}$ and $\tilde{W}|_{[0,t]}$ are measurable with respect to $\G_t$. Hence $\G_t$ is exactly the $\sigma$-algebra generated by $\tilde{u}|_{[0,t]}$ and $\tilde{W}|_{[0,t]}$.
    
    Let $\mathcal{N} \coloneqq \{ M \in \tilde{\F} \mid \tilde{\P}(M) = 0 \}$. We will consider the augmented filtration $(\tilde{\F}_t)_{t \in [0,\infty)}$ which is defined by
    \begin{align*}
      \tilde{\F}_t &\coloneqq \bigcap_{s > t} \sigma\left( \G_s \cup \mathcal{N} \right), ~t\in[0,\infty). 
    \end{align*}
    The augmented filtration $(\tilde{\F}_t)_t$ is a normal filtration.
    For $m \in \N$ we can do the same construction to define the natural filtration $(\G_t^m)_t$ and the augmented filtration $(\tilde{\F}_t^m)_t$ of $(\tilde{u}^m, \tilde{W}^m)$.

    We fix $k \in \N$ and define for $t\in[0,\infty)$
    \begin{align}
      \label{eq:variational-martingale}
      \begin{split}
      \tilde{M}(t) &\coloneqq \sprod[H]{\tilde{u}(t)}{e^k} - \sprod[H]{\tilde{u}(0)}{e^k}  - \int_0^t \sprod[V',V]{A(\tilde{u}(s))}{e^k} \ds \\
      \tilde{M}^m(t) &\coloneqq \sprod[H]{\tilde{u}^m(t)}{e^k} - \sprod[H]{\tilde{u}^m(0)}{e^k}  - \int_0^t \sprod[V',V]{A(\tilde{u}^m(s))}{e^k} \ds \\
      M^m(t) &\coloneqq \sprod[H]{u^{m}(t)}{e^k} - \sprod[H]{u^{m}(0)}{e^k}  - \int_0^t \sprod[V',V]{A(u^{m}(s))}{e^k} \ds.
      \end{split}
    \end{align}
    For $t \in [0,\infty)$ we have
    \begin{align*}
      M^m(t) = \int_0^t \sprod[H]{B(u^m(s)) \dW^m(s)}{e^k}.
    \end{align*}
    For $s \in [0,t]$ let
    \begin{align*}
      \gamma : C([0,s]; (H, w)) \times C([0,s]; (U_1, w)) \to \R
    \end{align*}
    be a bounded and continuous function.
    We will use the abbreviations
    \begin{align*}
      \gamma^m &\coloneqq \gamma\left( u^{m}|_{[0,s]}, W^m|_{[0,s]} \right),\\
      \tilde{\gamma}^m &\coloneqq \gamma\left( \tilde{u}^{m}|_{[0,s]}, \tilde{W}^m|_{[0,s]} \right), \\
      \tilde{\gamma} &\coloneqq \gamma\left( \tilde{u}|_{[0,s]}, \tilde{W}|_{[0,s]} \right).
    \end{align*}
    Since the joint law of $(\tilde{u}^m, \tilde{W}^m)$ coincides with the joint law of $(u^m, W^m)$, we infer for $l_1, l_2 \in \N$ and $m$ large enough that
    \begingroup
    \allowdisplaybreaks
    \begin{align}
      \label{eq:variational-wiener}
      \begin{split}
        0 &= \E \left( \gamma^m \left( W^m(t) - W^m(s) \right) \right) \\
        & = \tilde{\E} \left( \tilde{\gamma}^m \left( \tilde{W}^m(t) - \tilde{W}^m(s) \right) \right), \\
        \left(t-s\right) \delta_{l_1, l_2} a_{l_1}^2 &= {\E} \left( \gamma^m \sprod[U_1]{W^m(t)}{g_{l_1}}\sprod[U_1]{{W}^m(t)}{g_{l_2}} \right) \\
        &\phantom{{}={}} - {\E} \left(\gamma^m \sprod[U_1]{{W}^m(s)}{g_{l_1}}\sprod[U_1]{{W}^m(s)}{g_{l_2}} \right) \\
        &= \tilde{\E} \left( \tilde{\gamma}^m \sprod[U_1]{\tilde{W}^m(t)}{g_{l_1}}\sprod[U_1]{\tilde{W}^m(t)}{g_{l_2}}\right) \\
        &\phantom{{}={}} -\tilde{\E}\left( \tilde{\gamma}^m \sprod[U_1]{\tilde{W}^m(s)}{g_{l_1}}\sprod[U_1]{\tilde{W}^m(s)}{g_{l_2}} \right) \\
      \end{split}
    \end{align}
    and
    \begin{align}
      \label{eq:variational-quadvar}
      \begin{split}
        0 &= \E \left( \gamma^m \left( M^m(t) - M^m(s) \right) \right) \\
        &= \tilde{\E} \left( \tilde{\gamma}^m \left( \tilde{M}^m(t) - \tilde{M}^m(s) \right) \right), \\
        0 &= \E \left( \gamma^m \left( (M^m)^2(t) - (M^m)^2(s) - \int_s^t \left\| Q_m B^\ast(u^m(\sigma)) e^k \right\|_{U}^2 \dif \sigma \right) \right) \\
        &= \tilde{\E} \left( \tilde{\gamma}^m \left( (\tilde{M}^m)^2(t) - (\tilde{M}^m)^2(s) - \int_s^t \left\| Q_m B^\ast(\tilde{u}^m(\sigma)) e^k \right\|_{U}^2 \dif \sigma \right) \right), \\
        0 &= \E \left( \gamma^m \left( M^m(t)\sprod[U]{W^m(t)}{g_{l_1}} - M^m(s)\sprod[U]{W^m(s)}{g_{l_1}}\right)\right) \\
        &\phantom{{}={}} - \E\left(\gamma^m\int_s^t \sprod[H]{B(u^{m}(\sigma)) g_{l_1}}{e^k} \dif \sigma \right) \\
        &= \tilde{\E} \left( \tilde{\gamma}^m\left( \tilde{M}^m(t)\sprod[U]{\tilde{W}^m(t)}{g_{l_1}} - \tilde{M}^m(s)\sprod[U]{\tilde{W}^m(s)}{g_{l_1}} \right) \right) \\
        &\phantom{{}={}} -\tilde{\E} \left(\tilde{\gamma}^m \int_s^t \sprod[H]{B(\tilde{u}^{m}(\sigma)) g_{l_1}}{e^k} \dif \sigma \right).
      \end{split}
    \end{align}
    The Burkholder-Davis-Gundy inequality for $W^m$ yields the uniform bound
    \begin{align*}
      \tilde{\E} \| \tilde{W}^m(t) \|_{U_1}^3 = \E \| W^m(t) \|_{U_1} &\le C t^{\frac{3}{2}}.
    \end{align*}
    Now, with the Vitali convergence theorem we can pass to the limit in the equations \eqref{eq:variational-wiener} and infer
    \begin{align}
      \label{eq:variational-wiener-limit}
      \begin{split}
        0 &= \tilde{\E} \left( \tilde{\gamma} \left( \tilde{W}(t) - \tilde{W}(s) \right) \right), \\
        \left(t-s\right) \delta_{l_1, l_2} a_{l_1}^2 &= \tilde{\E} \left( \tilde{\gamma} \sprod[U_1]{\tilde{W}(t)}{g_{l_1}}\sprod[U_1]{\tilde{W}(t)}{g_{l_2}}\right) \\
        &\phantom{{}={}}- \tilde{\E}\left(\tilde{\gamma} \sprod[U_1]{\tilde{W}(s)}{g_{l_1}}\sprod[U_1]{\tilde{W}(s)}{g_{l_2}} \right). \\
      \end{split}
    \end{align}
    
    \endgroup
    Similarly, because of \Cref{lem:variational-continuity} and the convergence $Q_m \to \Id$ in $L(U)$, we conclude that in each of the above equations in \eqref{eq:variational-quadvar} we have the pointwise convergence of the variables for $m \to \infty$.  Furthermore, the Burkholder-Davis-Gundy inequality for $M^m$, the growth bound \eqref{eq:variational-growth-B-to-V'} and the estimates in \Cref{prop:variational-bounds} imply for some $q > 2$
    \begin{align*}
      \tilde{\E} |\tilde{M}^m(t)|^q = \E |M^m(t)|^q &\le C \E \left[ \int_0^t \left\| B^\ast(u^m(s)) e^k \right\|_U^2 \ds \right]^{\frac{q}{2}} \\
                                    &\le C\E \left[ \int_0^t \left\| B(u^m(s)) \right\|_{L(U;V')}^2 \ds \right]^{\frac{q}{2}} \| e^k \|_{V}^q \\
                                    &\le C_{k}\left(1 + \E \int_0^T \| u^m(s)\|_V^{2} \ds + \sup_{s \in [0,T]} \E \| u^m(s) \|_H^q \right) \\
                                    &\le C_{k,t}\left(1 + \E \| u^m(0) \|_H^q \right) \\
                                    &\le C_{k,t}\left(1 + \int_{H} \| z \|_H^q \dif\Lambda(z)\right).
    \end{align*}
    Again with the Vitali convergence theorem, we can pass to the limit in the equations \eqref{eq:variational-quadvar} and infer
    \begin{align}
      \label{eq:variational-quadvar-limit}
      \begin{split}
        0 &= \tilde{\E} \left( \tilde{\gamma} \left( \tilde{M}(t) - \tilde{M}(s) \right) \right), \\
        0 &= \tilde{\E} \left( \tilde{\gamma} \left( (\tilde{M})^2(t) - (\tilde{M})^2(s) - \int_s^t \left\| Q_m B^\ast(\tilde{u}(\sigma)) e^k \right\|_{U}^2 \dif \sigma \right) \right), \\
        0 &=  \tilde{\E} \left( \tilde{\gamma}\left( \tilde{M}(t)\sprod[U]{\tilde{W}(t)}{g_{l_1}} - \tilde{M}(s)\sprod[U]{\tilde{W}(s)}{g_{l_1}} \right) \right)\\
        &\phantom{{}={}} -\tilde{\E} \left(\tilde{\gamma} \int_s^t \sprod[H]{B(\tilde{u}(\sigma)) g_{l_1}}{e^k} \dif \sigma \right.
      \end{split}
    \end{align}
    Since the equations in \eqref{eq:variational-wiener-limit} hold for all $\gamma$, we conclude that $\tilde{W}$ is a square-integrable $(\G_t)_t$-martingale with $(\G_t)_t$-quadratic variation in $U$ given by
    \begin{align}
      \label{eq:variational-wiener-quadvar}
      \begin{split}
      \quadvar{ \tilde{W}(t) } &= tI.
      \end{split}
    \end{align}
    Since $\tilde{W}$ is continuous, we infer that $\tilde{W}$ is also a square-integrable $(\tilde{\F}_t)_t$-martingale and \eqref{eq:variational-wiener-quadvar} also holds for the quadratic variation with respect to $(\tilde{\F}_t)_t$. By the L\'evy martingale characterization \cite[Theorem 4.6]{DaPrato2014} we conclude that $\tilde{W}$ is a cylindrical $(\tilde{\F}_t)_t$-Wiener process on $U$.
    Similarly, as \eqref{eq:variational-quadvar-limit} holds for all $\gamma$, we conclude that $\tilde{M}$ is a square-integrable $(\G_t)_t$-martingale. Since $\tilde{M}$ is continuous by definition \eqref{eq:variational-martingale}, it is also a square-integrable $(\tilde{\F}_t)_t$-martingale. From \eqref{eq:variational-quadvar-limit} we also infer
    \begin{align}
      \label{eq:variational-quadvar-quadvar}
      \begin{split}
      \quadvar{ \tilde{M} - \int_0^\cdot \sprod[H]{B(\tilde{u}(s)) \dif\tilde{W}(s)}{e^k} } &= 0.
      \end{split}
    \end{align}
    Thus
    \begin{align*}
      \tilde{M}(t) - \tilde{M}(s) = \int_s^t \sprod[H]{B(\tilde{u}(\sigma)) \dW(\sigma)}{e^k} ~\tilde{\P}\text{-a.s.}
    \end{align*}
    for all $0 \le s \le t < \infty$ and $k \in \N$.  Furthermore we have
    \begin{align*}
      \Lambda \overset{\ast}{\leftharpoonup} \P \circ u^m(0)^{-1} = \tilde{\P} \circ \tilde{u}^m(0)^{-1} \overset{\ast}{\rightharpoonup} \tilde{\P} \circ \tilde{u}(0)^{-1}.
    \end{align*}
    Continuity of $\tilde{u}$ follows from \Cref{prop:variational-ito}.
  \end{proof}

  \begin{lemma}
    \label{lem:jakubowski-topological-assumption}
    Let $U$ be a separable Hilbert space. Then the following spaces have the property, that there is a countable set of real-valued continuous functions on this space that separates points:
    \begin{itemize}
    \item $C_{\text{loc}}([0,\infty);U)$,
    \item $L^2_{\text{loc}}(0,\infty; U)$,
    \item $(L^2_{\text{loc}}(0,\infty; U); w)$ and
    \item $C_{\text{loc}}([0,\infty); (U,w))$ with the compact-open topology.
    \end{itemize}
    \begin{proof}
      Fix dense and countable subsets $Q \subset [0,\infty)$ and $V \subset U$.
      Consider the set of functions $\mathfrak{F} = \{ u \mapsto \sprod[U]{u(q)}{v} \mid q \in Q, v \in V \}$. Then $\mathfrak{F}$ is a countable set of real-valued continuous functions on $C_{\text{loc}}([0,\infty);(U,w))$ that separates points. Since $C_{\text{loc}}([0,\infty); U) \subset C_{\text{loc}}([0,\infty); (U,w))$ is continuously embedded, the functions in $\mathfrak{F}$ are also continuous on $C_{\text{loc}}([0,\infty); U)$ and separate points.

      The space $L^2(0,\infty;U)$ is a separable Hilbert space. Let $\mathfrak{G} \subset C_c(0,\infty;U) \subset \left(L^2_{\text{loc}}(0,\infty;U)\right)^\ast$ be a countable set which is dense in $L^2(0,\infty;U)$. Then $\mathfrak{G}$ is a set of continuous functions on $L^2_{\text{loc}}(0,\infty;U)$ and $(L^2_{\text{loc}}(0,T;U),w)$ that separates points on both spaces.
    \end{proof}
  \end{lemma}

%%% Local Variables:
%%% mode: latex
%%% TeX-master: "main"
%%% End:

  \section{Matrix scalar product}
  \label{sec:matrix-sp}
  \begin{proposition}
    \label{prop:sp_matrices}
    Let $A, B, C \in \R^{n\times n}$ be symmetric matrices with $B, C \ge 0$. Then
    \begin{align*}
      AB : CA = (AB)_{i,j} (CA)_{i,j} \ge 0.
    \end{align*}

    \begin{proof}
      Write $B = D D^T$ and $C = E E^T$. Then
      \begin{align*}
        AB : CA &= ADD^T : EE^TA = |E^T A D|^2 \ge 0.
      \end{align*}
    \end{proof}
  \end{proposition}

\bibliography{literature}
\end{document}